\documentclass{amsart}

\usepackage[all,cmtip]{xy}

\usepackage{color,graphicx,amssymb,latexsym,amsfonts,txfonts,amsmath,amsthm}
\usepackage{pdfsync}
\usepackage{enumitem}

\usepackage{anyfontsize}

\usepackage{hyperref}
\hypersetup{
    colorlinks=true,       
    linkcolor=blue,          
    citecolor=blue,        
    filecolor=blue,      
    urlcolor=blue           
}

\newtheorem{theo}{Theorem}                     
\newtheorem{propo}{Proposition}                  
\newtheorem{coro}{Corollary} 
\newtheorem{lemm}{Lemma}
\theoremstyle{remark}                  
\newtheorem{rema}{\bf Remark}
\newtheorem{exem}{\bf Example}

\begin{document}

\title{Smooth quotients of generalized Fermat curves}

\author{Rub\'en A. Hidalgo}
\address{Departamento de Matem\'atica y Estad\'{\i}stica\\
Universidad de La Frontera, Temuco, Chile}
\email{ruben.hidalgo@ufrontera.cl}

\thanks{Partially supported by projects Fondecyt 1190001 and 1220261}

\begin{abstract} 
A closed Riemann surface  $S$ is called a generalized Fermat curve of type $(p,n)$, where $n,p \geq 2$ are integers such that $(p-1)(n-1)>2$, if it admits a group $H \cong {\mathbb Z}_{p}^{n}$ of conformal automorphisms with quotient orbifold $S/H$ of genus zero with exactly $n+1$ cone points, each one of order $p$; in this case $H$ is called a generalized Fermat group of type $(p,n)$. In this case, it is known that $S$ is non-hyperelliptic and that $H$ is its unique generalized Fermat group of type $(p,n)$. Also, explicit equations for them, as a fiber product of classical Fermat curves of degree $p$, are known. For  $p$ a prime integer, we describe those subgroups $K$ of $H$ acting freely on $S$, together with algebraic equations for $S/K$, and determine those $K$ such that $S/K$ is hyperelliptic. 
\end{abstract}

\maketitle

\section{Introduction}
If $S$ is a closed Riemann surface of genus $g \geq 2$, then the  group ${\rm Aut}(S)$, of its conformal automorphisms, is finite
\cite{Schwarz} and $|{\rm Aut}(S)| \leq 84(g-1)$ \cite{Hurwitz}. Moreover, every finite group can be realized as a group of conformal automorphisms of some closed Riemann surface of genus at least two \cite{Greenberg}. 

A group $H<{\rm Aut}(S)$ is called a generalized Fermat group of type $(p,n)$, where $n, p \geq 2$ are integers such that $(p-1)(n-1)>2$, if 
$H \cong {\mathbb Z}_{p}^{n}$ and the quotient orbifold ${\mathcal O}=S/H$ has genus zero and exactly $n+1$ cone points, each one of order $p$.
In this case, we say that $S$ (respectively, $(S,H)$) is a generalized Fermat curve (respectively, a generalized Fermat pair) of type $(p,n)$.
Generalized Fermat curves of type $(p,n)$ are the highest abelian (branched) covers of Riemann orbifolds ${\mathcal O}$ of genus zero with exactly $n+1$ cone points, each one of order $p$. 

Let  $(S,H)$ be a generalized Fermat pair of type $(p,n)$. By the Riemann-Hurwitz formula, the genus of $S$ is $g_{p,n}=1+p^{n-1}((n-1)(p-1)-2)/2 >1$.  In \cite{Hidalgo}, it was observed that $S$ is non-hyperelliptic and  that it is uniformized by the derived (commutator) subgroup of a Fuchsian group uniformizing the orbifold ${\mathcal O}=S/H$ and, 
in \cite{HKLP}, that $H$ is the unique generalized Fermat group of type $(p,n)$ of $S$. The uniqueness of $H$ permits us to think of generalized Fermat curves as a non-hyperelliptic version of hyperelliptic ones, where the generalized Fermat group $H$ replaces the role of the hyperelliptic involution. 
In \cite[Corollary 2]{GHL} (see also Sections \ref{Sec:algebra} and \ref{Fuchsiansetting}) it was observed that there is a subset $\{a_{1},\ldots,a_{n+1}\} \subset H$ such that: (i) each $a_{j}$ has order $p$, (ii) $a_{1}a_{2}\cdots a_{n+1}=1$, $H=\langle a_{1},\ldots,a_{n}\rangle$ and (iii) every element of $H$ acting with fixed points is a power of some $a_{j}$. We call $\{a_{1},\ldots,a_{n+1}\}$ a standard set of generators of $H$. 

As a consequence of the Riemann-Roch theorem, the surface $S$ can be represented by irreducible smooth complex algebraic curves. Simple algebraic equations for $S$ are given in  \cite{GHL} (see Section \ref{Sec:uniformiza}); these are provided by a suitable fiber product of classical Fermat curves of degree $p$. For instance, a generalized Fermat curve of type $(p,2)$ can be described by the classical Fermat curve of degree $p$.

Let us consider a Galois branched covering $S \to \widehat{\mathbb C}$ with deck group $H$ and with $n+1$ branch values. Up to postcomposition with a suitable M\"obius transformation, we may assume these branch values to be $p_{1}=\infty, p_{2}=0, p_{3}=1, p_{4}=\lambda_{1},\ldots, p_{n+1}=\lambda_{n-2}$, where $p_{j}$ is the projection of the fixed points of $a_{j}$. 
If $K$ is a subgroup of $H$ acting freely on $S$, then $R=S/K$ is a closed Riemann surface admitting the abelian group $G=H/K$ as a group of conformal automorphisms such that the quotient orbifold $R/G$ has genus zero and exactly $n+1$ cone points, each one of order $p$ (conversely, see Remark \ref{observa}, each such pair $(R,G)$ is obtained in that way). 

In Section \ref{Sec:libre}, for $p$ a prime integer, we describe the non-trivial subgroups $K$ of $H$ acting freely on $S$ (Proposition \ref{lema1}) together a corresponding algebraic curve for the underlying Riemann surface structure of $R=S/K$. In Section \ref{ejemplitos}, we provide explicit examples for $p=2$. In Section \ref{Sec:Prueba}, we obtain the following result, which describes those subgroups $K$ for which $R$ is hyperelliptic and also corresponding algebraic curves for them.  

\begin{theo}\label{explicito1}
Let $(S,H)$ be a generalized Fermat pair of type $(p,n)$, where $(p-1)(n-1)>2$ and $p$ is a prime integer. Let $\{a_{1},\ldots,a_{n+1}\}$ be a standard set of generators of $H$ and 
let $\pi:S \to \widehat{\mathbb C}$ be a Galois branched covering, with deck group $H$, whose branch values are 
$p_{1}=\infty, p_{2}=0, p_{3}=1, p_{4}=\lambda_{1},\ldots, p_{n+1}=\lambda_{n-2}$ and such that $p_{j}$ is the projection of the fixed points of $a_{j}$.

If $K$ is a subgroup of $H$ acting freely on $S$ such that $R=S/K$ is hyperelliptic, then one of the following holds.

\begin{enumerate}[label=(\alph*),leftmargin=*,align=left]

\item[{\bf (1)}] $p=2$, $n \geq 5$ is odd, $K=\langle a_{1}a_{2},a_{1}a_{3},\ldots,a_{1}a_{n}\rangle \cong {\mathbb Z}_{2}^{n-1}$ and  
$$R: \; y^{2}=x(x-1)(x-\lambda_{1})\cdots(x-\lambda_{n-2}).$$

\item[{\bf (2)}] $p=2$, $n \geq 4$ is even, $K=\langle a_{i_{1}} a_{i_{2}}, \ldots, a_{i_{1}} a_{i_{n-1}}\rangle \cong {\mathbb Z}_{2}^{n-2}$,
where $\{i_{1},\ldots, i_{n+1}\}=\{1,\ldots,n+1\}$, and
$$R: \; y^{2}=\prod_{j=1}^{2(n-1)}(x-\mu_{j}), \quad 
\{\mu_{1},...,\mu_{2(n-1)}\}=Q^{-1}(\{p_{i_{1}},\ldots,p_{i_{n-1}}\}),$$ 
$$
Q(z)=\left\{ \begin{array}{ll}
(p_{i_{n}}-p_{i_{n+1}}z^{2})/(1-z^{2}), & \mbox{ if $\infty \notin \{p_{i_{n}},p_{i_{n+1}}\}$}\\
z^{2}+p_{i_{n+1}}, & \mbox{if $p_{i_{n}}=\infty$}\\
p_{i_{n}}-z^{2}, & \mbox{if $p_{i_{n+1}}=\infty$}
\end{array}
\right\}.
$$

\item[{\bf (3)}] $p=2$, $n=5$,  $\lambda_{2}\lambda_{3}=\lambda_{1}$,  $K=\langle a_{i_{1}}a_{i_{2}}, a_{i_{3}}a_{i_{4}}, a_{i_{1}}a_{i_{3}}a_{i_{5}} \rangle \cong {\mathbb Z}_{2}^{3}$,
where $\{i_{1},i_{2},i_{3},i_{4},i_{5}\} \subset \{1,2,3,4,5,6\}$ and 
$$R:\; y^{2}=(x^{2}-a^{2})(x^{2}-a^{-2})(x^{2}-b^{2})(x^{2}-b^{-2}),$$
where $a^{2},b^{2} \in {\mathbb C} \setminus \{0,1,-1\}$ are such that
$$a^{2}+a^{-2}=M^{-1}\left(2\sqrt{\lambda_{1}}\right), \; b^{2}+b^{-2}=M^{-1}\left(-2\sqrt{\lambda_{1}}\right),$$
for $M(x)=\frac{1}{4}(1+\lambda_{1}-\lambda_{2}-\lambda_{3}) x +1+\lambda_{1}+\lambda_{2}+\lambda_{3}$.

\medskip
\item[{\bf (4)}] $p=2$, $n \geq 4$, $K=\langle a_{i_{1}}a_{i_{2}},a_{i_{1}}a_{i_{3}},\ldots,a_{i_{1}}a_{i_{n-2}}\rangle \cong {\mathbb Z}_{2}^{n-3}$, where $\{i_{1},\ldots, i_{n+1}\}=\{1,\ldots,n+1\}$,   and
$$R: \quad y^{2}=\prod_{j=1}^{n-2} (x^{4}+2(1-2L(p_{i_{j}}))x^{2}+1),$$
with $L$ the unique M\"obius transformation such that $L(\infty)=p_{i_{n-1}}$, $L(0)=p_{i_{n}}$ and $L(1)=p_{i_{n+1}}$.

\item[{\bf (5)}] $p \geq 3$ and $R$ is presented by either one of the following hyperelliptic curves:
\begin{enumerate} 
\item[(i)]  $y^{2}=x^{p}-1$ (of genus $(p-1)/2$), or
\item[(ii)]  $y^{2}=(x^{p}-1)(x^{p}-\alpha^{p})$ (of genus $p-1$), where $\sqrt{\alpha^{p}}=(\sqrt{\lambda_{1}}+1)/(\sqrt{\lambda_{1}}-1)$.
\end{enumerate}

\medskip

In both cases, (i) and (ii), the group $H/K \cong {\mathbb Z}_{p}$ is generated by $A(x,y)=(e^{2 \pi i/p}x,y)$. In case (i), $n=2$ and $K=\langle a_{j}a_{i}^{-1}\rangle$ ($i \neq j$). In case (ii), $n=3$ and $K=\langle a_{j}a_{i}^{-1}, a_{k}a_{i}^{-1}\rangle$ ($i,j,k$ pairwise distinct).
\end{enumerate}
\end{theo}

 In Section \ref{Ejemplo}, we work out explicitly the type $(2,4)$ (i.e., for classical Humbert curves). 
In Section \ref{prueba}, for type $(2,n)$, with $n \geq 4$, we provide an interpretation of the previous result in terms of maps between different moduli spaces
associated to the quotient hyperelliptic surfaces and the moduli space of generalized Fermat curves of type $(2,n)$ (Theorem \ref{main}).

\section{preliminaries}\label{Sec:uniformiza}
In this section we recollect some of the main properties on generalized Fermat pairs we will need in the rest of this paper.

Let $(S,H)$ be a generalized Fermat pair of type $(p,n)$, where  $n,p \geq 2$ are integers such that $(p-1)(n-1) > 2$.  

Let $\pi:S \to \widehat{\mathbb C}$ be a Galois branched covering map, with $H$ as its deck group, and whose branch values are given by the collection 
$${\mathcal B}=\{p_{1}=\infty, p_{2}=0, p_{3}=1, p_{4}=\lambda_{1},\ldots, p_{n+1}=\lambda_{n-2}\} \subset \widehat{\mathbb C}$$ (that choice is unique up to the action of ${\rm PSL}_{2}({\mathbb C})$).

\subsection{Algebraic models}\label{Sec:algebra}

\subsubsection{Some algebraic curves}
If $n=2$, then set $C^{p}(\varnothing)=\{x_{1}^{p}+x_{2}^{p}+x_{3}^{p}=0\} \subset {\mathbb P}^{2}$ (the classical Fermat curve of degree $p$).
If $n \geq 3$, then as the collection of points $\infty,0,1,\lambda_{1},\ldots,\lambda_{n-2}$ are pairwise different, the one dimensional algebraic set
{\small
\begin{equation}\label{eq2}
C^{p}(\lambda_{1},\ldots,\lambda_{n-2})=\left\{ \begin{array}{ccc}
x_{1}^{p}+x_{2}^{p}+x_{3}^{p}&=&0\\
\lambda_{1}x_{1}^{p}+x_{2}^{p}+x_{4}^{p}&=&0\\
\vdots \\
\lambda_{n-2}x_{1}^{p}+x_{2}^{p}+x_{n+1}^{p}&=&0
\end{array}
\right\} \subset {\mathbb P}^{n},
\end{equation}
}
is irreducible and smooth, so it represents a closed Riemann surface.

\subsubsection{Some linear automorphisms}\label{Sec:standard}
Note that each of the linear projective maps
$$a_{j}([x_{1}: \cdots :x_{n+1}])=[x_{1}: \cdots : x_{j-1} : e^{2\pi i/p}x_{j} : x_{j+1} : \cdots : x_{n+1}], \; j=1,\ldots, n+1,$$
is an order $p$  conformal automorphism of $C^{p}(\lambda_{1},\ldots,\lambda_{n-2})$ (for $n \geq 2$). 
Let us denote by $H_{0} \cong {\mathbb Z}_{p}^{n}$ the linear group generated by these automorphisms.
The map
$$\pi_{0}:C^{p}(\lambda_{1},\ldots,\lambda_{n-2}) \to \widehat{\mathbb C}: [x_{1}: \cdots : x_{n+1}] \mapsto -(x_{2}/x_{1})^{p}$$
turns out to be a Galois branched covering map, with deck group $H_{0}$, and 
whose branch locus is ${\mathcal B}$. Moreover, the fixed points of $a_{j}$ are projected under $\pi_{0}$ to the point $p_{j}$. In particular $(C^{p}(\lambda_{1},\ldots,\lambda_{n-2}),H_{0})$ is a generalized Fermat pair of type $(p,n)$.

\begin{propo}[\cite{GHL}]
There is a biholomorphism $\psi:S \to C^{p}(\lambda_{1},\ldots,\lambda_{n-2})$ such that $\pi=\pi_{0} \circ \psi$ (so, $\psi H \psi^{-1}=H_{0}$).
\end{propo}

\subsubsection{Standard set of generators}
It can be seen, from the explicit forms of the automorphisms $a_{j}$, that $a_{1}\cdots a_{n+1}=1$ and that the only elements of $H_{0}$, besides the identity, acting with fixed points on the curve are the powers of $a_{1},\ldots, a_{n+1}$.
As a consequence of the above proposition, we may find a collection $\{b_{1},\ldots,b_{n+1}\} \subset H$ such that: (i) each $b_{j}$ has order $p$, (ii) $b_{1}b_{2}\cdots b_{n+1}=1$, $H=\langle b_{1},\ldots,b_{n}\rangle$ and (iii) every element of $H$ acting with fixed points is a power of some $b_{j}$ (moreover, the fixed points of $b_{j}$ are projected under $\pi$ to $p_{j}$). We call such a set a  standard set of generators of $H$. Any other standard set of generators is of the form 
$\{b_{1}^{l},\ldots,b_{n+1}^{l}\}$, where $l \in \{1,\ldots,p-1\}$ is relatively prime to $p$. (Note that not every minimal set of generators of $H$ will be necessarily part of a standard set of generators.)

\begin{rema}[Connection to Petri's theorem]
\mbox{}
\begin{enumerate}[label=(\alph*),leftmargin=*,align=left]
\item[(1)] As the generalized Fermat curve $S$  is non-hyperelliptic, non-trigonal and not isomorphic to a plane quintic, 
Petri's theorem \cite{Petri,Saint} (together Noether's theorem) asserts that a canonical curve of $S$ (obtained by its embedding into ${\mathbb P}^{g_{p,n}-1}$ using a basis of its holomorphic one-forms) is an (not necessarily complete)  intersection of $(g_{p,n}-3)(g_{p,n}-2)/2$ quadric hypersurfaces. The fiber product representation $C^{p}(\lambda_{1},\ldots,\lambda_{n-2}) \subset {\mathbb P}^{n}$ is an embedding into a lower dimensional projective space for $(p,n)\neq (2,4)$. Nevertheless, in \cite{Hidalgo:bases} it was noted that $C^{p}(\lambda_{1},\ldots,\lambda_{n-2})$ is a projection of a suitable canonical curve obtained by forgetting some coordinates. 

\item[(2)] For $p=2$, so $n \geq 4$, the fiber product algebraic representation \eqref{eq2} is a complete intersection of $(n-1)$ quadrics in ${\mathbb P}^{n}$. Such types of complete intersections of $(n-1)$ quadrics in ${\mathbb P}^{n}$ where studied by Edge  in \cite{Edge2}. 

\item[(3)] Generalized Fermat curves of type $(2,4)$, also called {\it classical Humbert curves}, define a $2$-dimensional locus in the moduli space ${\mathcal M}_{5}$ (of genus five Riemann surfaces). These were firstly considered by Humbert in \cite{Humbert} as degree seven curves in ${\mathbb P}^{3}$. Later, in  \cite{Baker}, Baker rediscovered these curves related to a Weddle surface (see also \cite{Accola, Edge, Varley}).
In \cite{Edge}, Edge observed that each classical Humbert curve $S$ can be described as a complete intersection $C_{S}$ of three diagonal quadric hypersurfaces in ${\mathbb P}^{4}$: 
{\small
$$C_{S}=\left\{[x_{1}:x_{2}:x_{3}:x_{4}:x_{5}] \in {\mathbb P}^{4}: x_{1}^{2}+x_{2}^{2}+x_{3}^{2}+x_{4}^{2}+x_{5}^{2}=0,\right.$$
$$\left. \alpha_{1}x_{1}^{2}+\alpha_{2}x_{2}^{2}+\alpha_{3}x_{3}^{2}+\alpha_{4}x_{4}^{2}+\alpha_{5}x_{5}^{2}=0, \; \alpha_{1}^{2}x_{1}^{2}+\alpha^{2}_{2}x_{2}^{2}+\alpha_{3}^{2}x_{3}^{2}+\alpha_{4}^{2}x_{4}^{2}+\alpha_{5}^{2}x_{5}^{2}=0\right\},$$
}
where $\alpha_{1},\ldots,\alpha_{5} \in {\mathbb C}$ are different (in this model, $H=H_{0}$). This is isomorphic to the fiber product {\small
$$C^{2}(\lambda_{1},\lambda_{2})=\left\{[x_{1}:x_{2}:x_{3}:x_{4}:x_{5}] \in {\mathbb P}^{4}: x_{1}^{2}+x_{2}^{2}+x_{3}^{2}=0, \; \lambda_{1} x_{1}^{2}+x_{2}^{2}+x_{4}^{2}=0, \; \lambda_{2}x_{1}^{2}+x_{2}^{2}+x_{5}^{2}=0\right\},$$
}
where $\lambda_{1}$ and $\lambda_{2}$ are given in terms of $\alpha_{1},\ldots,\alpha_{5}$.
\end{enumerate}
\end{rema}
 
\subsection{Fuchsian uniformizations}\label{Fuchsiansetting}
By the uniformization theorem, there is a co-compact Fuchsian group $\Gamma_{p,n}$, of conformal automorphisms of the hyperbolic plane  ${\mathbb H}^{2}$, with a presentation
$\Gamma_{p,n}=\langle x_{1},\ldots, x_{n+1}: x_{1}^{p}=\cdots=x_{n+1}^{p}=x_{1}x_{2}\cdots x_{n+1}=1\rangle$, such that the orbifold ${\mathcal O}=S/H$ is conformally equivalent (as orbifolds) to  ${\mathbb H}^{2}/\Gamma_{p,n}$.  Due to results of Maclachlan \cite{Maclachlan}, its derived subgroup  $\Gamma'_{p,n}$ is torsion free, so the quotient $X_{p,n}={\mathbb H}^{2}/\Gamma'_{p,n}$ is a closed Riemann surface  with $H_{p,n}=\Gamma_{p,n}/\Gamma'_{p,n} \cong {\mathbb Z}_{p}^{n} < {\rm Aut}(X_{p,n})$ such that $X_{p,n}/H_{p,n}={\mathbb H}^{2}/\Gamma_{p,n}=S/H$. In particular,  $(X_{p,n}, H_{p,n})$ is a generalized Fermat pair of type $(p,n)$. In this setting, the classes in $H_{p,n}$ of the generators $x_{1}, \ldots, x_{n+1}$ provide a  standard set of generators of $H_{p,n}$.

\begin{propo}[\cite{GHL}] 
There is a biholomorphism $\phi:S \to X_{p,n}$ conjugating $H$ onto $H_{p,n}$.  
\end{propo}

\begin{rema}\label{observa}
The above Fuchsian description of generalized Fermat curves, in terms of the derived subgroup, permits to observe the following fact. If $R$ is a closed Riemann surface admitting an abelian group $G$ of conformal automorphisms such that the Riemann orbifold $R/G$ is given by the Riemann sphere, its cone points set is $\{\infty,0,1,\lambda_{1},\ldots,\lambda_{n-2}\}$, and the cone order of each of these cone points is $p$, then there is a subgroup $K$ of $H_{0}$, acting freely on $C^{p}(\lambda_{1},\ldots,\lambda_{n-2})$, such that $R=C^{p}(\lambda_{1},\ldots,\lambda_{n-2})/K$ and $G=H_{0}/K$.
\end{rema}

\subsection{Automorphisms}
Let $M<{\rm PSL}_{2}({\mathbb C})$ be the (finite) group of those M\"obius transformations keeping invariant the collection $\{\infty,0,1,\lambda_{1},\ldots,\lambda_{n-2}\}$. As $H$ is a normal subgroup of ${\rm Aut}(S)$ (due to its uniqueness), there is a natural homomorphism $\theta:{\rm Aut}(S) \to M$, whose kernel is $H$. As there is a biholomorphism $\phi:S \to X_{p,n}$ conjugating $H$ onto $H_{p,n}$, and $X_{p,n}$ is uniformized by the derived subgroup $\Gamma'_{p,n}$ (a characteristic subgroup of $\Gamma_{p,n}$), every (orbifold) automorphism of the orbifold $S/H$ lifts to an automorphism of $S$, so  $\theta$ is also surjective. In this way,  we obtain a short exact sequence
$$1 \to H \to {\rm Aut}(S) \stackrel{\theta}{\to} M \to 1,$$ which permits to compute explicitly ${\rm Aut}(S)$ (see \cite{GHL}).

\subsection{Moduli of generalized Fermat curves}\label{Sec:moduli}
As already mentioned, a generalized Fermat curve of type $(p,2)$ is algebraically described by the classical Fermat curve $x_{1}^{p}+x_{2}^{p}+x_{3}^{p}=0$ (in the projective plane) and it has no moduli. On the other hand, generalized Fermat curves of type $(p,n)$, where $n \geq 3$, form a $(n-2)$-dimensional family and they have the explicit algebraic models $C^{p}(\lambda_{1},\ldots,\lambda_{n-2})$, where $(\lambda_{1},\ldots,\lambda_{n-2})$ belongs to the connected open set 
{\small
$$V_{n}=\{(\lambda_{1},\ldots,\lambda_{n-2}) \in {\mathbb C}^{n-2}: \lambda_{1},\ldots,\lambda_{n-2} \in {\mathbb C} \setminus \{0,1\}, \; \lambda_{j} \neq \lambda_{r}, j \neq r\} \subset {\mathbb C}^{n-2}.$$ 
}

\begin{rema}
In \cite{GHL}, the above region was denoted by the symbol ${\mathcal P}_{n}$.
\end{rema}

The region $V_{n}$ admits the following holomorphic automorphisms:
{\small
$$t(\lambda_{1},\ldots,\lambda_{n-2})=\left(\frac{\lambda_{n-2}}{\lambda_{n-2}-1}, \frac{\lambda_{n-2}}{\lambda_{n-2}-\lambda_{1}},\ldots,\frac{\lambda_{n-2}}{\lambda_{n-2}-\lambda_{n-3}}\right), \;
b(\lambda_{1},\ldots,\lambda_{n-2})=\left(\frac{1}{\lambda_{1}},\ldots,\frac{1}{\lambda_{n-2}}\right).$$
}

Let us denote by ${\mathfrak S}_{m}$ the permutation group on $m$ points and set ${\mathbb G}_{n}=\langle t,b\rangle$.

Let $\sigma \in {\mathfrak S}_{n+1}$ and $(\lambda_{1},\ldots,\lambda_{n-2}) \in V_{n}$. If we set $p_{1}=\infty, p_{2}=0, p_{3}=1, p_{4}=\lambda_{1}, \ldots, p_{n+1}=\lambda_{n-2}$, then 
we may consider the unique M\"obius transformation $M_{\sigma}$ such that $M_{\sigma}(p_{\sigma^{-1}(1)})=\infty$, $M_{\sigma}(p_{\sigma^{-1}(2)})=0$ and $M_{\sigma}(p_{\sigma^{-1}(3)})=1$. If, for $j \geq 4$, we set $M_{\sigma}(p_{\sigma^{-1}(j)})=\lambda_{j-3}^{\sigma}$, then $(\lambda_{1}^{\sigma},\ldots, \lambda_{n-1}^{\sigma}) \in V_{n}$. In this way, we obtain a bijective map 
$$\Theta(\sigma): V_{n} \to V_{n}: (\lambda_{1},\ldots, \lambda_{n-1}) \mapsto (\lambda_{1}^{\sigma},\ldots, \lambda_{n-1}^{\sigma}).$$

One may check that $\Theta((1,2))=b$ and $\Theta((1,2,\ldots,n+1))=t$. In particular, this allowd us to note that 
$\Theta:{\mathfrak S}_{n+1} \to {\mathbb G}_{n}$ is a surjective homomorphism with 
${\mathbb G}_{3} \cong {\mathfrak S}_{3}$ and, for $n \geq 4$, ${\mathbb G}_{n} \cong {\mathfrak S}_{n+1}$ (details can be found in \cite{Yasmina}).

\begin{propo}[Section 4 in \cite{GHL}] \label{proposicion3}
If $\vec{\lambda}, \vec{\delta} \in V_{n}$, then  
$C^{k}(\vec{\lambda})$ and $C^{k}(\vec{\delta})$
 are conformally equivalent
 if and only if $\vec{\lambda}$ and $\vec{\delta}$ belong to the same ${\mathbb G}_{n}$-orbit. In particular, the quotient orbifold $V_{n}/{\mathbb G}_{n}$ is a model for the moduli space ${\mathcal H}_{p,n}$ of generalized Fermat curves of type $(p,n)$.
 \end{propo}

\section{Description of acting freely subgroups of generalized Fermat groups}\label{Sec:libre}
In this section, we assume $p$ is a prime integer, $S=C^{p}(\lambda_{1},\ldots,\lambda_{n-2})$ (or $C^{p}(\varnothing)$ if $n=2$), $H=H_{0}$ and 
$\{a_{1},\ldots,a_{n+1}\}$ is the set of standard generators as previously described in Section \ref{Sec:algebra}. 
In Section \ref{sinfijos}, we provide a description of those non-trivial subgroups $K$ of $H$ acting freely on $S$ (Proposition \ref{lema1}) and, in Section \ref{Sec:curvas}, we describe an algebraic curve representing $S/K$.

In the following, for each $r \geq 1$, we fix an ordering of the elements of ${\mathbb Z}_{p}^{r}$, say ${\mathbb Z}_{p}^{r}=\{u_{0}=1,u_{1},\ldots,u_{p^{r}-1}\}$.

\subsection{Freely acting subgroups of $H$}\label{sinfijos}
For each $r \in \{1,\ldots,n-1\}$, we let
${\mathcal F}_{r,p}^{n}$ be the collection of tuples  $(I_{1},\ldots,I_{p^{r}-1})$, where $\{I_{1},\ldots,I_{p^{r}-1}\}$ is a (disjoint) partition of the set $\{1,\ldots,n+1\}$ (some of the $I_{j}$ might be the empty set) such that: 
\begin{enumerate}[label=(\alph*),leftmargin=*,align=left]
\item[(i)] $u_{1}^{\#I_{1}}\cdots u_{p^{r}-1}^{\#I_{p^{r}-1}}=1$ and 
\item[(ii)] $\langle u_{k}: k \in \{1,\ldots,p^{r}-1\}, \;  I_{k}\neq \varnothing\rangle={\mathbb Z}_{p}^{r}$. 
\end{enumerate}

\begin{rema}
As $n+1 =\sum_{k=1}^{p^{r}-1} \#I_{k}$, it follows that, for $p=2$ (so, $n \geq 4$) and $r \in \{1,2\}$, at least one of the $I_{k}$ has cardinality at least two. It might be that some ${\mathcal F}_{r,p}^{n}$ is the empty set.
\end{rema}

Set $A_{n,p}=\{r \in \{1,\ldots,n-1\}:  {\mathcal F}_{r,p}^{n} \neq \varnothing\}$.

If $r \in A_{n,p}$ and $P=(I_{1},\ldots,I_{p^{r}-1}) \in {\mathcal F}_{r,p}^{n}$, then condition (i) asserts that there is a homomorphism
$\rho_{r,P}:H \to {\mathbb Z}_{p}^{r}$, defined by $\rho_{r,P}(a_{j})=u_{k}$ ($j \in I_{k}$). Condition (ii) asserts that $\rho_{r,P}$ is surjective.
If we denote by $K_{P} \cong {\mathbb Z}_{p}^{n-r}$ its kernel, then (as the only elements of $H$ acting with fixed points are the powers of $a_{1},\ldots,a_{n+1}$) it can be seen that $K_{P}$ acts freely on $S$. 
This provides one direction of the next description.

\begin{propo}\label{lema1} Let $\{I\} \neq K<H$.  Then $K$ acts freely on $S$ if and only if there are some $r \in A_{n,p}$ and some $P \in {\mathcal F}_{r,p}^{n}$ such that 
$K=K_{P}$.
\end{propo}
\begin{proof}
Above we have already seen one direction. Now, let $K \notin \{ \{I\}, H\}$ be a subgroup of $H$ acting freely on $S$.
The subgroup $K$ is the kernel of a suitable surjective homomorphism $\rho:H \to {\mathbb Z}_{p}^{r}=\{u_{0}=1,u_{1},\ldots,u_{p^{r}-1}\}$, for some $1 \leq r \leq n-1$.

As the only elements of $H$ acting with fixed points on $S$ are the powers of its standard generators, $K$ is torsion free if and only if $\rho(a_{j}) \neq 1$, for every $j=1,\ldots,n+1$.

Set $I_{k}=\{j \in \{1,\ldots,n+1\}: \rho(a_{j})=u_{k}\}$, for $k=1,\ldots,p^{r}-1$. We note that $K$ is torsion free if
 $\{I_{1},\ldots,I_{p^{r}-1}\}$ is a (disjoint) partition of the set $\{1,\ldots,n+1\}$. Some of the $I_{k}$ might be the empty set. 
 Also, as $a_{1}\cdots a_{n+1}=1$, we must have $u_{1}^{\#I_{1}}\cdots u_{p^{r}-1}^{\#I_{p^{r}-1}}=1$, and the surjectivity condition asserts $\langle u_{k}: I_{k}\neq \varnothing, k \in \{1,\ldots,p^{r}-1\}\rangle={\mathbb Z}_{p}^{r}$. It follows that $r \in A_{n,p}$, that $P=(I_{1},\ldots,I_{p^{r}-1}) \in {\mathcal F}_{r,p}^{n}$, that $\rho=\rho_{r,P}$ and that $K=K_{P}$.
\end{proof}

\begin{coro}\label{coro:m=n-1}
If $(p,n)=(2,n)$, where $n \geq 4$, and $K \cong {\mathbb Z}_{2}^{n-1}$ is a subgroup of $H$, acting freely on $S$, then $n$ is odd. Moreover, in this situation, $K=\langle a_{1}a_{2}, a_{1}a_{3}, \ldots, a_{1}a_{n+1}\rangle$ and $S/K$ is the hyperelliptic Riemann surface of genus $(n-1)/2$ defined by the curve $y^{2}z^{n-2}=x(x-z)\prod_{j=1}^{n-2}(x-\lambda_{j}z)$.
\end{coro}
\begin{proof}
As $r=1$, we are just considering the partition $I_{1}=\{1,\ldots,n+1\}$ and ${\mathbb Z}_{2}=\{1,u_{1}\}$. The condition $u_{1}^{\#I_{1}}=u_{1}^{n+1}=1$ is equivalent for $n$ to be odd. Now, under this condition on $n$, we obtain $K=\langle a_{1}a_{2}, a_{1}a_{3}, \ldots, a_{1}a_{n+1}\rangle$. As the Riemann surface $S/K$ is a two-fold branched cover of $S/H$, we obtain that $S/K$ is a hyperelliptic curve as described.
\end{proof}

\subsection{Algebraic curves for $S/K$}\label{Sec:curvas}
Let $K<H$ be a freely acting subgroup. By Proposition \ref{lema1}, there is some $r \in A_{n,p}$ and some 
$P=(I_{1},\ldots,I_{p^{r}-1}) \in {\mathcal F}_{r,p}^{n}$, such that $K=K_{P}$. Classical geometric invariant theory \cite{Mumford} permits to construct an algebraic curve model for $S/K_{P}$. For it, consider the affine model $S^{0}$, obtained by setting $x_{n+1}=1$ in $C^{p}(\lambda_{1},\ldots,\lambda_{n-2})$, which is invariant under $H$. Next, we obtain a set of generators of the algebra of invariants ${\mathbb C}[x_{1},\ldots,x_{n}]^{K_{P}}$ (which is known to be finitely generated by Hilbert-Noether's theorem \cite{Noether1}). As the elements of $K_{P}$ are diagonal matrices, we may find such a set of generators formed by monomials. 
As a non-trivial element of $K_{P}$ is a diagonal matrix, whose diagonal is formed by values equal to powers of $e^{2 \pi i/p}$, some 
of these  invariant monomials are given by $t_{1}=x_{1}^{p},\ldots, t_{n}=x_{n}^{p}$,  and others will be of the form $s_{k}=x_{1}^{l_{1,k}}x_{2}^{l_{2,k}}\cdots x_{n}^{l_{n,k}}$, for suitables $l_{j,k} \in \{0,1,\ldots,p-1\}$, $k=1,\ldots,m_{P}$. 
In this case, if $\psi:S^{0} \to {\mathbb C}^{n+m_{P}}$ is defined by $\psi(x_{1},\ldots,x_{n})=(t_{1},\ldots,t_{n},s_{1},\ldots,s_{m_{P}})$, then the (possible singular) curve $\psi(S^{0})\subset {\mathbb C}^{n+m_{P}}$, which is defined by
{\small
\begin{equation}\label{curvita}
\left\{\begin{array}{c}
t_{1}+t_{2}+t_{3}=0, \; \lambda_{1}t_{1}+t_{2}+t_{4}=0, \ldots, \lambda_{n-3}t_{1}+t_{2}+t_{n}, \; \lambda_{n-2}t_{1}+t_{2}+1=0,\\
s_{k}^{p}=\prod_{j =1}^{n} t_{j}^{l_{j,k}}, \; k=1,\ldots,m_{P}.
\end{array}\right\},
\end{equation}
}
produces an affine model for $S^{0}/K_{P}$ (this being an affine model for $S/K_{P}$).  From the first line (of the above set of equations), we observe that 
{\small
\begin{equation}\label{reemplaza}
\begin{array}{c}
t_{2}=t_{2}(t_{1})=-(1+\lambda_{n-2}t_{1}), \; t_{3}=t_{3}(t_{1})=1+(\lambda_{n-2}-1)t_{1},\\
t_{4}=t_{4}(t_{1})=1+(\lambda_{n-2}-\lambda_{1})t_{1}, \ldots, t_{n}=t_{n}(t_{1})=1+(\lambda_{n-2}-\lambda_{n-3})t_{1}.
\end{array}
\end{equation}
}

In this way, if we consider the projection map
$$Q:(t_{1},\ldots,t_{n},s_{1},\ldots,s_{m_{P}}) \in {\mathbb C}^{n+m_{P}} \mapsto (t_{1},s_{1},\ldots,s_{m_{P}}) \in {\mathbb C}^{1+m_{P}},$$
then we observe that $Q(\psi(S^{0}))$ is also an affine model of $S^{0}/K_{P}$.
If we now replace each $t_{2},\ldots, t_{n}$, in terms of the variable $t_{1}$ (as in \eqref{reemplaza}), in the equations on the second line of \eqref{curvita}, then this last affine model is given by the following $p$-cyclic gonal curve: 
{\small
\begin{equation}
\left\{\begin{array}{c}
s_{k}^{p}=\prod_{j=1}^{n}t_{j}^{l_{j,k}}; \; k=1,\ldots,m_{P}.
\end{array}\right\} \subset {\mathbb C}^{1+m_{P}}.
\end{equation}
}

\subsection{Some examples: ${\bf (p,n)=(2,n)}$, ${\bf n \geq 4}$}\label{ejemplitos}

\subsubsection{}
Let $n \geq 4$ and $r=2$. If $P=(I_{1},I_{2},I_{3}) \in {\mathcal F}_{2,2}^{n}$, where $\#I_{1}=n-1$, then 
(since $\#I_{1}+\#I_{2}+\#I_{3}=n+1$) we have the following two cases.

\begin{enumerate}[label=(\alph*),leftmargin=*,align=left]

\item If $n \geq 4$ is even, then $\#I_{2}=\#I_{3}=1$. Up to permutation of indices, we may assume 
$$I_{1}=\{1,\ldots,n-1\}, \; I_{2}=\{n\}, \; I_{3}=\{n+1\}.$$

\item If $n \geq 5$ is odd, then either (i) $\#I_{2}=2$ and  $\#I_{3}=0$ or (ii) $\#I_{2}=0$ and  $\#I_{3}=2$. Let us consider the case (i) (the other is similar). Up to permutation of indices, we may assume 
$$I_{1}=\{1,\ldots,n-1\}, \; I_{2}=\{n,n+1\}, \; I_{3}=\varnothing.$$

\end{enumerate}

In any of the above two cases, $K_{P}=\langle a_{1}a_{2},a_{1}a_{3},\ldots,a_{1}a_{n-1}\rangle \cong {\mathbb Z}_{2}^{n-2}$. The $K_{P}$-invariant monomials are $t_{1}=x_{1}^{2},\ldots,t_{n}=x_{n}^{2}$, $s_{1}=x_{1}x_{2}\cdots x_{n-1}$ and $s_{2}=x_{n}$. In this way, $S/K_{P}$ is the hyperelliptic Riemann surface (of genus $n-2$) defined by
{\small
$$S/K_{P}: \left\{\begin{array}{l}
s_{1}^{2}= -t_{1}(1+\lambda_{n-2}t_{1})(1+(\lambda_{n-2}-1)t_{1})\prod_{j=1}^{n-4}(1+(\lambda_{n-2}-\lambda_{j})t_{1})\\
s_{2}^{2}=1+(\lambda_{n-2}-\lambda_{n-3})t_{1}
\end{array}
\right\}.
$$
}
and  $H/K_{P}=\langle  A(t_{1},s_{1},s_{2})=(t_{1},-s_{1},s_{2}), B(t_{1},s_{1},s_{2})=(t_{1},s_{1},-s_{2})  \rangle \cong {\mathbb Z}_{2}^{2}$, where $A$ corresponds to the hyperelliptic involution.

\subsubsection{}
Let $n=5$, $r=2$ and $P=(I_{1},I_{2},I_{3}) \in {\mathcal F}_{2,2}^{5}$, where $\#I_{1}=\#I_{2}=\#I_{3}=2$. Up to permutation of the indices, we may assume 
$I_{1}=\{1,2\}, \; I_{2}=\{3,4\}, \; I_{3}=\{5,6\}.$
In this case, $K_{P}=\langle a_{1}a_{2}, a_{3}a_{4}, a_{1}a_{3}a_{5}\rangle \cong {\mathbb Z}_{2}^{3}$. The $K_{P}$-invariant monomials are $t_{1}=x_{1}^{2},\ldots,t_{5}=x_{5}^{2}$, $s_{1}=x_{1}x_{2}x_{5}$, $s_{2}=x_{3}x_{4}x_{5}$ and $s_{3}=x_{1}x_{2}x_{3}x_{4}$. In this way, $S/K_{P}$ is the Riemann surface of genus three defined by
{\small
$$S/K_{P}:\left\{\begin{array}{l}
s_{1}^{2}=-t_{1}(1+\lambda_{3}t_{1})(1+(\lambda_{3}-\lambda_{2})t_{1})\\
s_{2}^{2}=(1+(\lambda_{3}-1)t_{1})(1+(\lambda_{3}-\lambda_{1})t_{1})(1+(\lambda_{3}-\lambda_{2})t_{1})\\
s_{3}^{2}=-t_{1}(1+\lambda_{3}t_{1})(1+(\lambda_{3}-1)t_{1})(1+(\lambda_{3}-\lambda_{1})t_{1})
\end{array}
\right\}
$$
}
and $H/K_{P}=\langle  A(t_{1},s_{1},s_{2},s_{3})=(t_{1},-s_{1},s_{2},-s_{3}), B(t_{1},s_{1},s_{2},s_{3})=(t_{1},s_{1},-s_{2},-s_{3})   \rangle \cong {\mathbb Z}_{2}^{2}$. In this case, $A$, $B$ and $AB$ each one has $4$ fixed points.

In the particular case $\lambda_{3}=\lambda_{1}/\lambda_{2}$, it happens that $C^{2}(\lambda_{1},\lambda_{2},\lambda_{3})$ also admits the extra conformal involution
$T([x_{1}:\cdots:x_{6}])= [x_{2}:\sqrt{\lambda_{1}} x_{1}: x_{4}: \sqrt{\lambda_{1}} x_{3}: \sqrt{\lambda_{2}} x_{6}: \sqrt{\lambda_{3}} x_{5}],$
which satisfies $\pi \circ T=\tau \circ \pi$, where $\tau(x)=\lambda_{1}/x$. If we choose the square roots such that $\sqrt{\lambda_{1}}=\sqrt{\lambda_{2}} \sqrt{\lambda_{3}}$, then $T^{2}=I$. As $T$ normalizes $K_{P}$, it induces a conformal involution $E$ of $S/K_{P}$, which is given by
$E(t_{1},s_{1},s_{2},s_{3})=(\lambda_{2} t_{2}/\lambda_{1}t_{5}, \lambda_{2}^{2}s_{1}/\lambda_{1} t_{5}^{2}, \lambda_{2}^{2} s_{2}/\lambda_{1} t_{5}^{2}, \lambda_{2}^{2}s_{3}/\lambda_{1} t_{5}^{2})$, where 
$t_{2}=-(1+\lambda_{3}t_{1})$ and $t_{5}=1+(\lambda_{3}-\lambda_{2})t_{1}$. It can be checked that $E$ is the hyperelliptic involution of $S/K_{P}$.

\subsubsection{}
Let $n \geq 4$, $r=3$, ${\mathbb Z}_{2}^{3}=\langle u_{1}, u_{2}, u_{3}\rangle$ and $u_{4}=u_{1}u_{2}$, $u_{5}=u_{2}u_{3}$, $u_{6}=u_{1}u_{3}$ and $u_{7}=u_{1}u_{2}u_{3}$. If $P=(I_{1},\ldots,I_{7}) \in {\mathcal F}_{3,2}^{n}$, where $\#I_{1}=n-2$, then we have the following cases to consider.

\begin{enumerate}[label=(\alph*),leftmargin=*,align=left]

\item If $n \geq 4$ is even,  $\#I_{2}=\#I_{3}=\#I_{5}=1$ and $\#I_{4}=\#I_{6}=\#I_{7}=0$. Up to permutation of indices, we may assume 
$I_{1}=\{1,\ldots,n-2\}, \; I_{2}=\{n-1\}, \; I_{3}=\{n\}, \; I_{5}=\{n+1\}.$

\item If $n \geq 5$ is odd, $\#I_{2}=\#I_{3}=\#I_{7}=1$ and  $\#I_{4}=\#I_{5}=\#I_{6}=0$. Up to permutation of indices, we may assume 
$I_{1}=\{1,\ldots,n-2\}, \; I_{2}=\{n-1\}, \; I_{3}=\{n\}, \; I_{7}=\{n+1\}.$

\end{enumerate}

In any of these two cases, $K_{P}=\langle a_{1}a_{2},a_{1}a_{3},\ldots,a_{1}a_{n-2}\rangle \cong {\mathbb Z}_{2}^{n-3}$. The $K_{P}$-invariant monomials are $t_{1}=x_{1}^{2},\ldots,t_{n}=x_{n}^{2}$, $s_{1}=x_{1}x_{2}\cdots x_{n-2}$, $s_{2}=x_{n-1}$ and $s_{3}=x_{n}$. In this way, $S/K_{P}$ is the hyperelliptic Riemann surface (of genus $2n-5$) defined by
{\small
$$
S/K_{P}: \left\{\begin{array}{l}
s_{1}^{2}= -t_{1}(1+\lambda_{n-2}t_{1})(1+(\lambda_{n-2}-1)t_{1})\prod_{j=1}^{n-5}(1+(\lambda_{n-2}-\lambda_{j})t_{1})\\
s_{2}^{2}=1+(\lambda_{n-2}-\lambda_{n-4})t_{1}\\
s_{3}^{2}=1+(\lambda_{n-2}-\lambda_{n-3})t_{1}
\end{array}
\right\}
$$
}
and $H/K_{P}=\langle A(t_{1},s_{1},s_{2},s_{3})=(t_{1},-s_{1},s_{2},s_{3}), B(t_{1},s_{1},s_{2},s_{3})=(t_{1},s_{1},-s_{2},s_{3}), \\ C(t_{1},s_{1},s_{2},s_{3})=(t_{1},s_{1},s_{2},-s_{3}) \rangle \cong {\mathbb Z}_{2}^{3}$, where $A$ corresponds to the hyperelliptic involution.

\subsection{An application to hyperelliptic quotients}\label{Sec:casohiperp=2}
Let $R$ be a hyperelliptic Riemann surface and $G \cong {\mathbb Z}_{2}^{r}$ be a group of its conformal automorphisms such that $R/G$ has genus zero. Each cone point of $R/G$ has order two and the number of them is $n+1$ for some $n \geq 4$.
 Up to a M\"obius transformation, we may assume these cone points to be $\infty,0,1,\lambda_{1},\ldots,\lambda_{n-2}$. As noted in Remark \ref{observa}, there is a 
subgroup $K \cong {\mathbb Z}_{2}^{m}$ of $H$, acting freely on $S=C^{(2)}(\lambda_{1},\ldots,\lambda_{n-2})$, such that $R=S/K$ and $G=H/K$, i.e., $m=n-r$.

\begin{lemm}\label{cocientes}
In the above.
\begin{enumerate}[label=(\alph*),leftmargin=*,align=left]

\item If $n$ is odd, then $m \in \{n-3,n-2,n-1\}$. 
\item If $n$ is even, then $m \in \{n-3,n-2\}$.
\end{enumerate}
\end{lemm}
\begin{proof}
As $H \cong {\mathbb Z}_{2}^{n}$ acts with fixed points, $m \in \{1,...,n-1\}$. Note that, for $n \geq 4$ even, $H$ has no subgroup isomorphic to ${\mathbb Z}_{2}^{n-1}$ acting freely on $S$ (otherwise, there will be a closed Riemann surface admitting a conformal involution with an odd number of fixed points, a contradiction to the Riemann-Hurwitz formula).
In particular, for $n \geq 4$ even, $m \leq n-2$. 

The group $H/K \cong {\mathbb Z}_{2}^{n-m}$ is a group of conformal automorphisms of the hyperelliptic Riemann surface $R$. If $\iota$ denotes its hyperelliptic involution, then either (i) $\iota \in H/K$ or (ii) $\iota \notin H/K$. In the first case, $H/K$ induces an action of
${\mathbb Z}_{2}^{n-m-1}$ as a group of M\"obius transformations. In the second case, $H/K$ induces an action of
${\mathbb Z}_{2}^{n-m}$ as a group of M\"obius transformations.
As the only Abelian subgroups of ${\rm PSL}_{2}({\mathbb C})$ are cyclic or ${\mathbb Z}_{2}^{2}$, the above asserts $m \in \{n-3,n-2,n-1\}$.
\end{proof}

\begin{rema}
Let $(S,H)$ be a generalized Fermat pair of type $(2,n)$. If $n \geq 5$ is odd, then there is exactly one subgroup $K \cong {\mathbb Z}_{2}^{n-1}$ of $H$ acting freely on $S$ and, by Corollary \ref{coro:m=n-1}, $S/K$ is hyperelliptic.
If $n \geq 4$, then subgroups of $H$ isomorphic to ${\mathbb Z}_{2}^{n-2}$ or ${\mathbb Z}_{2}^{n-3}$ and acting freely on $S$ may not be unique and they may not necessarily produce hyperelliptic quotients as can be seen in the following example. 
\end{rema}

\begin{exem}[Example with non-hyperelliptic quotients]
Next, for $n=7$, we construct a subgroup $K \cong {\mathbb Z}_{2}^{n-2}$, acting freely on $S$, such that $S/K$ is non-hyperelliptic (similar examples for $K \cong {\mathbb Z}_{2}^{n-3}$ can be produced). Let $a_{1},\ldots,a_{8} \in H$ be a standard set of generators. Consider the surjective homomorphism $\rho:H \to G=\langle u,v: u^{2}=v^{2}=(uv)^{2}=1\rangle \cong {\mathbb Z}_{2}^{2}$ defined by $\rho(a_{i})=u$ and $\rho(a_{j})=v$, where $i \in \{1,2,3,4\}$ and $j \in \{5,6,7,8\}$. This corresponds to $P=(I_{1}=\{1,2,3,4\},I_{2}=\{5,6,7,8\}, I_{3}=\varnothing) \in {\mathcal F}_{2,2}^{7}$, where $u_{1}=u$, $u_{2}=v$ and $u_{3}=uv$, and $\rho=\rho_{7,P}$.
The kernel $K$ of $\rho$ acts freely on $S$ and $R=S/K$ is a closed Riemann surface of genus five on which the group $G$ acts as a group of conformal automorphisms with $S/H=R/G$. The involutions $u$ and $v$ have, each one, exactly $8$ fixed points and the involution $uv$ acts freely on $S$. We claim that $R$ is non-hyperelliptic. In fact, if $R$ is hyperelliptic, then its  hyperelliptic involution $\iota$ has $12$ fixed points, so $\iota \notin G$. By projecting $\iota$ to the quotient orbifold $R/G$, we see that the induced involution must have exactly two fixed points and the $8$ cone points are permuted into $4$ pairs. It follows that $R/\langle G,\iota\rangle$ has genus zero with exactly $6$ cone points of order two; exactly $4$ of them being the projection of the fixed points of the elements of $G$. Now, by projecting on the orbifold $R/\langle \iota \rangle$, the group $G$ induces an isomorphic M\"obius group that permutes the $12$ cone points (and fixing no one of them). This asserts that $R/\langle G, \iota \rangle$ must be of genus zero and with exactly $8$ cone points; where $3$ of them are the projections of the fixed points of $G$; a contradiction with the previous. 
\end{exem}

The following result, which will be needed it in the proof of Theorem \ref{main}, concerns with 
subgroups of $H$ isomorphic to ${\mathbb Z}_{2}^{n-2}$ and acting freely on $S$.

\begin{lemm}\label{teo5}
Let $(S,H)$ be a generalized Fermat pair of type $(2,n)$, where $n \geq 4$. If $K_{1}$ and $K_{2}$ are two subgroups of $H$, both isomorphic to ${\mathbb Z}_{2}^{n-2}$ and acting freely on $S$, then both pairs $(S/K_{1},H/K_{1})$ and $(S/K_{2},H/K_{2})$ are conformally equivalent if and only if there is some $f \in {\rm Aut}(S)$ so that $f K_{1} f^{-1}=K_{2}$.
\end{lemm}
\begin{proof} One direction is clear by the uniqueness of $H$. On the other direction, 
if there is a conformal homeomorphism $\phi:S/K_{1} \to S/K_{2}$ so that $\phi (H/K_{1}) \phi^{-1} = H/K_{2}$, then $\phi$ induces a conformal automorphism $\psi$ of $S/H=(S/K_{j})/(H/K_{j})$. As $S$ is the homology cover of $S/H$, this means that $\psi$ lifts to a conformal automorphism $f \in {\rm Aut}(S)$ that conjugates $K_{1}$ to $K_{2}$.
\end{proof}

\begin{rema}
In the generic situation one has that ${\rm Aut}(S)=H$ is abelian (recall that ${\rm Aut}(S)/H$ is isomorphic to the group $M$ of M\"obius transformations keeping invariant the set of conical points of $S/H$). In that generic case, for $n \geq 4$, the above result asserts that if  $K_{1}$ and $K_{2}$ are two different subgroups of $H$, both isomorphic to ${\mathbb Z}_{2}^{n-2}$ and acting freely on $S$, then $(S/K_{1},H/K_{1})$ and $(S/K_{2},H/K_{2})$ are not conformally equivalent.
\end{rema}

\section{Proof of Theorem \ref{explicito1}}\label{Sec:Prueba}
Let $S=C^{p}(\lambda_{1},\ldots,\lambda_{n-2})$ (or $C^{p}(\varnothing)$ for $n=2$), where $p \geq 2$ is a prime integer such that $(p-1)(n-1)>2$,
$H=H_{0}=\langle a_{1},\ldots,a_{n}\rangle$ be its generalized Fermat group of type $(p,n)$ (as in Section \ref{Sec:standard}) and  $\pi=\pi_{0}:S \to \widehat{\mathbb C}$ be the 
Galois (branched) covering with deck group $H$ and whose branch values are the points 
$p_{1}=\infty, p_{2}=0, p_{3}=1, p_{4}=\lambda_{1},\ldots, p_{n+1}=\lambda_{n-2}$, where $p_{j}$ is the $\pi$-projection of the fixed points of the standard generator $a_{j}$ (where $a_{n+1}=(a_{1}\cdots a_{n})^{-1}$).

Let $K \neq \{I\}$ be a subgroup  of $H$ acting freely on $S$ such that $R=S/K$ is hyperelliptic. We denote by $\iota$ the hyperelliptic involution of $R$.
Then $K \cong {\mathbb Z}_{p}^{m}$, for some $m \in \{1,\ldots,n-1\}$, where $r=n-m \in A_{n,p}$ (Proposition \ref{lema1}). The group $G=H/K \cong {\mathbb Z}_{p}^{r}$ is a group of conformal automorphisms of $R$ (note that $\iota$ might or not belong to $G$).
If $P:R \to \widehat{\mathbb C}$ is a Galois (branched) covering with deck group $\langle \iota \rangle$, then an algebraic curve representation of $R$ is provided by the curve
$$y^{2}=\prod_{j=1}^{2g+2}(x-q_{j}),$$
where $q_{1},\ldots,q_{2g+2}$ are the branch values of $P$ (if some of them equals to $\infty$, then we delete the corresponding factor $(x-q_{j})$).

Let us consider: (i) a Galois covering $F:S \to R$, with deck group $K$, and (ii) a Galois (branched) covering $\widehat{\pi}:R \to \widehat{\mathbb C}$, with deck group $G$, such that $\pi=\widehat{\pi} \circ F$. 

The Galois (branched) covering $R \to \widehat{\mathbb C}$, with deck group $\langle G, \iota\rangle$, factors through $P$ and 
$\widehat{\pi}$, that is, there is a Galois (branched) covering $Q^{*}:\widehat{\mathbb C} \to \widehat{\mathbb C}$, with deck group $\langle G, \iota \rangle/G$ (so, if $\iota \in G$, then we may assume $Q^{*}$ to be the identity map) and there is a Galois (branched) covering $Q:\widehat{\mathbb C} \to \widehat{\mathbb C}$, with deck group $\langle G, \iota\rangle/\langle \iota \rangle$, such that $Q^{*} \circ \widehat{\pi}=Q \circ P$. To obtain the equations of $R$, we need to describe $Q$. We proceed to divide our study into two cases: (1) $p=2$ and (2) $p \geq 3$.

\subsection{Proof of parts (1), (2), (3) and (4): The case ${\bf p=2}$ (so ${\bf n \geq 4}$)}
By Lemma \ref{cocientes}, $r \in \{1,2,3\}$.
Let us consider a Galois (unbranched) covering $S \to R$, induced by the action of $K$. As $S$ has genus $1+2^{n-2}(n-3)$, the Riemann-Hurwitz formula asserts that $R$ has genus $g_{R}=1+2^{r-2}(n-3)$.

\subsubsection{{\bf Part (1):} ${\bf r=1}$} 
In this case, $K\cong {\mathbb Z}_{2}^{n-1}$ and 
Corollary \ref{coro:m=n-1} asserts that $n \geq 5$ is odd and that 
$$K=\langle a_{1}a_{2},a_{1}a_{3},\ldots,a_{1}a_{n}\rangle, \quad  S/K= \; y^{2}=x(x-1)(x-\lambda_{1})\cdots(x-\lambda_{n-2}).$$ 

\subsubsection{{\bf Parts (2) and (3):} ${\bf r=2}$}
In this case, $K\cong {\mathbb Z}_{2}^{n-2}$, $R$ has genus $n-2$ and  
$G=\{1,u_{1}=u, u_{2}=v, u_{3}=uv\}=\langle u, v: u^{2}=v^{2}=(uv)^{2}=1\rangle$. By Proposition \ref{lema1}, there is some $P=(I_{1},I_{2},I_{3}) \in {\mathcal F}_{2,2}^{n}$ such that $K$ is the kernel  of a homomorphism $\rho_{2}:H \to G$ defined by  
$$\rho_{2}(a_{j})=\left\{\begin{array}{ll}
u, & j \in I_{1},\\
v, & j \in I_{2},\\
uv, & j \in I_{3}.
\end{array}
\right.
$$

Since $1=u^{\#I_{1}}v^{\#I_{2}}(uv)^{\#I_{3}}=u^{\#I_{1}+\#I_{3}}v^{\#I_{2}+\#I_{3}}$, that is, 
the sets $I_{1}, I_{2}, I_{3}$ have cardinalities of the same parity (i.e., all of them odd or all of them even).
The group $G\cong {\mathbb Z}_{2}^{2}$ acts as a group of conformal automorphisms of $R$ such that $R/G=S/H$ is the Riemann sphere with exactly $n+1$ cone points of order two. The automorphism $u$ has $2\#I_{1}$ fixed points, $v$ has $2\#I_{2}$ fixed points and $uv$ has $2\#I_{3}$ fixed points.  Also, $\#I_{1}+\#I_{2}+\#I_{3}=n+1$. 
We have two possibilities to consider: (i) $\iota \in G$ or (ii) $\iota \notin G$.

\subsubsection*{\bf (i) Case ${\bf \iota \in G}$}
One of the elements of $G$ must be $\iota$, so one of $I_{j}$ must have cardinality $n-1$ (as $R$ has genus $n-2$). 
As  $\#I_{1}+\#I_{2}+\#I_{3}=n+1$ and the three have the same parity, we have the following.
 \begin{enumerate}[label=(\alph*),leftmargin=*,align=left]
 
\item[(a)] If $n \geq 4$ is even, then either: (i) $\#I_{1}=n-1$ and $\#I_{2}=\#I_{3}=1$ or (ii) $\#I_{2}=n-1$ and $\#I_{1}=\#I_{3}=1$ or (iii) $\#I_{3}=n-1$ and $\#I_{1}=\#I_{2}=1$.
\item[(b)] If $n \geq 5$ is odd, then either: (i) $\#I_{1}=n-1$, $\#I_{2}=2$ and $\#I_{3}=0$ or (ii) $\#I_{1}=n-1$, $\#I_{2}=0$ and $\#I_{3}=2$ or
(iii) $\#I_{2}=n-1$, $\#I_{1}=2$ and $\#I_{3}=0$ or (iv) $\#I_{2}=n-1$, $\#I_{1}=0$ and $\#I_{3}=2$ or (v) $\#I_{3}=n-1$, $\#I_{1}=2$ and $\#I_{2}=0$ or (vi) $\#I_{3}=n-1$, $\#I_{1}=0$ and $\#I_{2}=2$.
\end{enumerate}

The above permits to see that the corresponding collection of subgroups $K$ of $H$ (a collection of cardinality $n(n+1)/2$) are of the form
$$K_{P}=K(\{i_{1}, \ldots,i_{n-1}\})=\langle a_{i_{1}} a_{i_{2}}, \ldots, a_{i_{1}} a_{i_{n-1}}\rangle \cong {\mathbb Z}_{2}^{n-2},$$
where 
$\{i_{1}, \ldots, i_{n-1}\} \subset \{1,\ldots,n+1\}$ corresponds to the component set $I_{j}$ of $P$ of cardinality $n-1$. 
Let $b_{1}=p_{i_{n}}$ and $b_{2}=p_{i_{n+1}}$.

In this case, without lost of generality, we may assume $G/\langle \iota \rangle=\langle z \mapsto -z\rangle$. So, the (branched) covering $Q:\widehat{\mathbb C} \to \widehat{\mathbb C}$ has its critical values at $0$ and $\infty$ and it satisfies that $Q(-z)=Q(z)$ and $\{Q(\infty),Q(0)\}=\{b_{1},b_{2}\}$.
Then the set of Weierstrass points of $S/K(\{i_{1}, \ldots, i_{n-1}\})$ is given by the set 
$$\{\mu_{1},\ldots,\mu_{2(n-1)}\}=Q^{-1}(\{p_{i_{1}},\ldots,p_{i_{n-1}}\}),$$ an explicit equation for it is given by
{\small
$$S/K(\{i_{1}, \ldots, i_{n-1}\})= \; y^{2}=\prod_{j=1}^{2(n-1)}(x-\mu_{j}).$$
}

For example, we may take
$$
Q(z)=\left\{ \begin{array}{ll}
(b_{1}-b_{2}z^{2})/(1-z^{2}), & \mbox{ if $\infty \notin \{b_{1},b_{2}\}$}\\
z^{2}+b_{2}, & \mbox{if $b_{1}=\infty$}\\
b_{1}-z^{2}, & \mbox{if $b_{2}=\infty$}
\end{array}
\right\}.
$$

\subsubsection*{\bf (ii) Case ${\bf \iota \notin G}$} We will observe that $n=5$.
As $\iota \notin G$, then $\iota$ does not share a fixed point with any involution of $G$. By 
projecting $\iota$ to $R/G$, we obtain a conformal involution $\widehat{\iota}$ that permutes in pairs the $n+1$ cone points $\{\infty,0,1,\lambda_{1},\ldots, \lambda_{n-2}\}$ and fixes none of them. It follows that $n$ is odd. Up to a M\"obius transformation, we may assume that $\widehat{\iota}$ permutes $\infty$ with $0$, $1$ with $\lambda_{1}$ (so $\widehat{\iota}(x)=\lambda_{1}/x$) and $\lambda_{2j}$ with $\lambda_{2j+1}$, for $j=1,\ldots, (n-3)/2$. This asserts that on the genus zero quotient $R/\langle G, \iota \rangle$ we have only two cone points coming from the fixed points of $\iota$ and the others $(n+1)/2$ from the fixed points of $G$.
We may now consider a Galois branched cover of degree two induced by $\iota$, say $T:R \to \widehat{\mathbb C}$, so that $G$ induces, under $T$, the group $\widehat{G}=\langle A(x)=-x, B(x)=1/x\rangle$. The $2n-2$ branch values of $T$ are permuted by $\widehat{G}$, none of them being fixed by an involution on it. In particular, the above asserts that the five cone points of the quotient orbifold $R/\langle \iota,G\rangle$ are divided into two disjoint sets; one of cardinality three which are the ones coming from the fixed points of the involutions of $G$ and the other of cardinality $(2n-2)/4$ which are coming from the fixed points of $\iota$.
Since $2=(2n-2)/4$ and $(n+1)/2=3$, it holds that $n=5$. All the above asserts that $S$ has genus $17$,
that $R$ has genus three and that $G \cong {\mathbb Z}_{2}^{2}$ acts on $R$ with quotient orbifold $R/G$ being the sphere with exactly $6$ cone points, each one of order two, given by $\infty$, $0$, $1$, $\lambda_{1}$, $\lambda_{2}$ and $\lambda_{3}=\lambda_{1}/\lambda_{2}$. 

It is known that a hyperelliptic Riemann surface of genus three admitting a group of conformal automorphisms $G \cong {\mathbb Z}_{2}^{2}$ (which does not contains the hyperelliptic involution) can be described by a curve of the form
$$y^{2}=(x^{2}-a^{2})(x^{2}-a^{-2})(x^{2}-b^{2})(x^{2}-b^{-2}),$$
for suitable values of $a$ and $b$ such that $a^{2},b^{2} \in {\mathbb C}\setminus \{0,1,-1\}$. In this model, the group $G$ is generated by $u(x,y)=(-x,y)$ and $v(x,y)=(1/x,y/x^{4})\rangle$.
In order to get the values of these parameters $a$ and $b$, in terms of the values $\lambda_{1}$ and $\lambda_{2}$, we may proceed as follows. As before, we assume that the involution $\widehat{\iota}$, which is induced by $\iota$ on the orbifold $R/G$, is given by $\widehat{\iota}(x)=\lambda_{1}/x$. The map $L(x)=x+\lambda_{1}/x$ provides a two-fold branched covering with deck group $\langle \; \widehat{\iota} \; \rangle$. The cone values of $R/\langle G, \iota \rangle$ are provided by the two branch values of $L$ (these are given by the points $\pm 2\sqrt{\lambda_{1}}$) and the image of the values $\infty, 0, 1, \lambda_{1}, \lambda_{2}$ and $\lambda_{3}$ (these are given by the points $\infty$, $1+\lambda_{1}$ and $\lambda_{2}+\lambda_{1}/\lambda_{2}$). A Galois branched covering, with deck group $\widehat{G}$, is given by $Q_{1}(x)=x^{2}+x^{-2}$. The three branch values of $Q_{1}$ are given by the points $\pm 2$ and $\infty$, and it sends the values $\pm a$, $\pm a^{-1}$, $\pm b$ and $\pm b^{-1}$ to $a^{2}+a^{-2}$ and $b^{2}+b^{-2}$. If we consider the M\"obius transformation $M(x)=\alpha x +\beta$, where
$\alpha=\frac{1}{4}(1+\lambda_{1}-\lambda_{2}-\lambda_{3})$ and $\beta=1+\lambda_{1}+\lambda_{2}+\lambda_{3}$, then 
$$Q_{2}(x)=M \circ Q_{1}(x)=\frac{1}{4}(1+\lambda_{1}-\lambda_{2}-\lambda_{3})(x^{2}+x^{-2})+\frac{1}{2}(1+\lambda_{1}+\lambda_{2}+\lambda_{3})$$ 
provides a degree four Galois branched covering, with deck group $\widehat{G}$, whose three branch values are $\infty$, $1+\lambda_{1}$ and $\lambda_{2}+\lambda_{1}/\lambda_{2}$.  It sends the points $\pm a$, $\pm a^{-1}$, $\pm b$ and $\pm b^{-1}$ to $\pm 2\sqrt{\lambda_{1}}$.
In this way, the choice of the values of $a$ and $b$ are given by the property that 
$$M(a^{2}+a^{-2})=2\sqrt{\lambda_{1}}, \; M(b^{2}+b^{-2})=-2\sqrt{\lambda_{1}}.$$

This situation corresponds to $I_{1}=\{i_{1},i_{2}\}$, $I_{2}=\{i_{3},i_{4}\}$, $I_{3}=\{i_{5},i_{6}\}$ and the group $K=\langle a_{i_{1}}a_{i_{2}}, a_{i_{3}}a_{i_{4}},a_{i_{1}}a_{i_{3}}a_{i_{5}} \rangle \cong {\mathbb Z}_{2}^{3}$, 
where $\{1,2,3,4,5,6\}=\{i_{1},i_{2},i_{3},i_{4},i_{5},i_{6}\}$.

\subsubsection{{\bf Part (4):} ${\bf r=3}$}
In this case, $K\cong {\mathbb Z}_{2}^{n-3}$, $R$ has genus $2n-5$ and 
$G=\{1,u_{1}=u, u_{2}=v, u_{3}=w, u_{4}=uv, u_{5}=vw, u_{6}=uw, u_{7}=uvw\}=\langle u,v,w: u^{2}=v^{2}=w^{2}=(uv)^{2}=(uw)^{2}=(vw)^{2}=1\rangle$. 
By Proposition \ref{lema1},  there is some $P=(I_{1},\ldots,I_{7}) \in {\mathcal F}_{3,2}^{n}$ such that $K$ is the kernel  of a homomorphism $\rho_{3}:H \to G$ defined by  
$$\rho_{3}(a_{j})=\left\{\begin{array}{ll}
u, & j \in I_{1},\\
v, & j \in I_{2},\\
w, & j \in I_{3},\\
uv, & j \in I_{4},\\
vw, & j \in I_{5},\\
uw, & j \in I_{6},\\
uvw, & j \in I_{7}.
\end{array}
\right.
$$

Since $1=u^{\#I_{1}} v^{\#I_{2}} w^{\#I_{3}} (uv)^{\#I_{4}} (vw)^{\#I_{5}} (uw)^{\#I_{6}} (uvw)^{\#I_{7}}=u^{\#I_{1}+\#I_{4}+\#I_{6}+\#I_{7}}
v^{\#I_{2}+\#I_{4}+\#I_{5}+\#I_{7}} w^{\#I_{3}+\#I_{5}+\#I_{6}+\#I_{7}}$, 
we must have that
$$(*)\left\{ \begin{array}{c}
\#I_{1}+\#I_{4}+\#I_{6}+\#I_{7} \equiv 0 \mod(2),\\
\#I_{2}+\#I_{4}+\#I_{5}+\#I_{7} \equiv 0 \mod(2),\\
\#I_{3}+\#I_{5}+\#I_{6}+\#I_{7} \equiv 0 \mod(2),
\end{array}
\right.
$$

In this case, the group $G=H/K \cong {\mathbb Z}_{2}^{3}$ acts as a group of conformal automorphisms of $R$ such that $R/G=S/H$. The involutions $u,v,w,uv,vw,uw,uvw$ have, respectively, $4\#I_{1},4\#I_{2},4\#I_{3},4\#I_{4},4\#I_{5},4\#I_{6},4\#I_{7}$ fixed points.  
We claim that its hyperelliptic involution $\iota$ must belong to $G$. In fact, otherwise $G \cong {\mathbb Z}_{2}^{3}$ must induce an isomorphic group of M\"obius transformations on the quotient $R/\langle \iota \rangle$, a contradiction to the fact that the only finite abelian subgroups of ${\rm PSL}_{2}({\mathbb C})$ are the cyclic ones and ${\mathbb Z}_{2}^{2}$. 
We may assume that $u=\iota$, that is, $\#I_{1}=n-2$; so $\#I_{2}+\#I_{3}+\#I_{4}+\#I_{5}+\#I_{6}+\#I_{7}=3$. It follows from this and  $(*)$ that 
$$\begin{array}{ll}
(a) & n-2+\#I_{4}+\#I_{6}+\#I_{7} \equiv 0 \mod(2),\\
(b) & \#I_{2}+\#I_{4}+\#I_{5}+\#I_{7}\in \{0,2\}\\
(c) & \#I_{3}+\#I_{5}+\#I_{6}+\#I_{7} \in \{0,2\}.
\end{array}
$$

If $\#I_{2}+\#I_{4}+\#I_{5}+\#I_{7} =0$, then (from (c)) $\#I_{3}+\#I_{6}\in \{0,2\}$, which contradicts the fact that  $\#I_{2}+\#I_{3}+\#I_{4}+\#I_{5}+\#I_{6}+\#I_{7}=3$. Similarly, 
if $\#I_{3}+\#I_{5}+\#I_{6}+\#I_{7} =0$, then it again provides a contradiction. In this way,
$$\begin{array}{ll}
(i) & n-2+\#I_{4}+\#I_{6}+\#I_{7} \equiv 0 \mod(2),\\
(ii) & \#I_{2}+\#I_{3}+\#I_{4}+\#I_{5}+\#I_{6}+\#I_{7}=3,\\
(iii) & \#I_{2}+\#I_{4}+\#I_{5}+\#I_{7} =2,\\
(iv) & \#I_{3}+\#I_{5}+\#I_{6}+\#I_{7} =2.\\
\end{array}
$$

It follows, by combining (ii) and (iii) and then (ii) with (iv), that $\#I_{3}+\#I_{6}=1=\#I_{2}+\#I_{4}$.
Then, by (ii) one also has that $\#I_{5}+\#I_{7}=1$ and, by (i) that  $n-2+\#I_{4}+\#I_{6}+\#I_{7}$ is even. As (by (ii)) 
$\#I_{4}+\#I_{6}+\#I_{7} \in \{0,1,2,3\}$, we observe that, 
for $n$ even, $\#I_{4}+\#I_{6}+\#I_{7}\in \{0,2\}$ and, for $n$ odd, $\#I_{4}+\#I_{6}+\#I_{7} \in \{1,3\}$.

Summarizing all the above:
\begin{enumerate}[label=(\alph*),leftmargin=*,align=left]

\item If $n \geq 4$ is even, then $\#I_{1}=n-2$, and either:
\begin{enumerate}
\item $\#I_{4}=\#I_{6}=\#I_{7}=0$ and $\#I_{2}=\#I_{3}=\#I_{5}=1$.
\item $\#I_{4}=\#I_{3}=\#I_{5}=0$ and $\#I_{2}=\#I_{6}=\#I_{7}=1$.
\item $\#I_{2}=\#I_{5}=\#I_{6}=0$ and $\#I_{3}=\#I_{4}=\#I_{7}=1$.
\item $\#I_{2}=\#I_{3}=\#I_{7}=0$ and $\#I_{4}=\#I_{5}=\#I_{6}=1$.

\end{enumerate}

\item If $n \geq 5$ is odd, then $\#I_{1}=n-2$, and either:
\begin{enumerate}
\item  $\#I_{2}=\#I_{6}=\#I_{7}=0$ and $\#I_{4}=\#I_{3}=\#I_{5}=1$.
\item  $\#I_{3}=\#I_{4}=\#I_{7}=0$ and $\#I_{2}=\#I_{5}=\#I_{6}=1$.
\item $\#I_{4}=\#I_{5}=\#I_{6}=0$ and $\#I_{2}=\#I_{3}=\#I_{7}=1$. 
\item  $\#I_{2}=\#I_{3}=\#I_{5}=0$ and $\#I_{4}=\#I_{6}=\#I_{7}=1$.
\end{enumerate}
\end{enumerate}

In each of the above cases (a)-(d), for either $n$ even or odd, we have $n(n^{2}-1)/6$ possible tuples $P=(I_{1},\ldots,I_{7}) \in {\mathcal F}_{3}^{n}$, and for each of them we have the corresponding group $K=K_{P}$. 

Let us assume  $I_{1}=\{i_{1},\ldots,i_{n-2}\}$. In this case, 
$K_{P}=\langle a_{i_{1}}a_{i_{2}},a_{i_{1}}a_{i_{3}},\ldots,a_{i_{1}}a_{i_{n-2}}\rangle$. Let $b_{1}=p_{i_{n-1}}, b_{2}=p_{i_{n}}$ and $b_{3}=p_{i_{n+1}}$.
In this case, without lost of generality, we may assume $G/\langle \iota \rangle=J=\langle z \mapsto -z, z \mapsto 1/z\rangle \cong {\mathbb Z}_{2}^{2}$. So, the Galois (branched) covering $Q$ 
has deck group $J$ whose branch values are $b_{1}, b_{2}$ and $b_{3}$.

Let $T$ be the (unique) M\"obius transformation such that $T(b_{1})=\infty$, $T(b_{2})=0$ and $T(b_{3})=1$, that is,
$$T(z)=\left\{ \begin{array}{ll}
(z-b_{2})(b_{3}-b_{1})/(z-b_{1})(b_{3}-b_{2}), & \mbox{if  $\infty \notin \{b_{1},b_{2},b_{3}\}$}\\
(z-b_{2})/(b_{3}-b_{2}), & \mbox{if  $\infty=b_{1}$}\\
(b_{3}-b_{1})/(z-b_{1}), & \mbox{if  $\infty = b_{2}$}\\
(z-b_{2})/(z-b_{1}), & \mbox{if  $\infty=b_{3}$}
\end{array}
\right\}.
$$

The map $U(z)=((1+z^{2})/2z)^{2}$ is a Galois branched covering with deck group $J=\langle z \mapsto -z, z \mapsto 1/z\rangle \cong {\mathbb Z}_{2}^{2}$ and whose branch values are $\infty$, $0$ and $1$. So, if $Q=T^{-1} \circ U$, then 
$Q:\widehat{\mathbb C} \to \widehat{\mathbb C}$ is a Galois branched cover, with deck group $J$, whose branch values are $b_{1}$, $b_{2}$ and $b_{3}$, as requiered.

The $Q$-preimage of the set $\{p_{i_{1}},\ldots,p_{i_{n-2}}\}$ consists of $4n-8$ points, a disjoint union of $n-2$ $J$-obits (the preimages of each $p_{i_{j}}$). These lifted points are the Weierstrass points of the hyperelliptic Riemann surface $S/K_{P}$.
As the four $Q$-preimages of a point $p_{i_{j}} \in \widehat{\mathbb C} \setminus \{b_{1},b_{2},b_{3}\}$ are the zerores of the polynomial $x^{4}+2(1-2q_{i_{j}})x^{2}+1$, where $q_{i_{j}}=T^{-1}(p_{i_{j}})$, an equation of $S/K_{P}$ has the form
{\small
$$S/K_{P}: \quad y^{2}=\prod_{j=1}^{n-2} (x^{4}+2(1-2q_{i_{j}})x^{2}+1).$$
}
\subsection{Proof of part (5) of Theorem \ref{explicito1}: The case $p \geq 3$}
As $\iota \notin G$, the group $G$ induces an abelian group $\widehat{G} \cong {\mathbb Z}_{p}^{r}$, of M\"obius transformations keeping invariant the projections of the fixed points of $\iota$ (these are $2(g_{p,n}+1)$ points). As the only finite abelian subgroups of ${\rm PSL}_{2}({\mathbb C})$ are either cyclic ones or the Klein group ${\mathbb Z}_{2}^{2}$, it follows that $r=1$.  Moreover, the cyclic group $G$ acts with exactly $n+1$ fixed points.
The action of $\widehat{G} \cong {\mathbb Z}_{p}$ on the orbifold $S/\langle \iota \rangle$ produces the orbifold $R/\langle G, \iota \rangle$ of signature either: (1) $(0;2,\stackrel{q}{\ldots},2,p,p)$ or (2) $(0;2,\stackrel{q}{\ldots},2,2p,2p)$ or (3) $(0;2,\stackrel{q}{\ldots},2,p,2p)$.
On the other hand, the involution $\iota$ induces a conformal involution of the orbifold $R/G=S/H$, and its action on $S/H$ produces the orbifold $R/\langle G, \iota \rangle$ with signature either:
(a) $(0;2,2,p,\stackrel{l}{\ldots},p)$ (so $n+1=2l$) or  (b) $(0;2p,2p,p,\stackrel{l}{\ldots},p)$ (so $n+1=2l+2$) or (c) $(0;2,2p,p,\stackrel{l}{\ldots},p)$ (so $n+1=2l+1$).
Combining all the above, we observe that the only possibilities are:
\begin{enumerate}
\item[(I)] case: (1) and (a) with $l=q=2$ (so $n=3$), or 
\item[(II)] case: (3) and (c) with $q=l=1$ (so $n=2$).
\end{enumerate}

In case (II), the surface $R$ is represented by the hyperelliptic curve as in part (5) (i)  of Theorem \ref{explicito1}. 

In case (I), it is represented by the hyperelliptic curve as described in part (5) (ii) of Theorem \ref{explicito1}. 
Note that the quotient orbifold $R/\langle \iota, A\rangle$ can be identified with $\widehat{\mathbb C}$ and its branch values to be: $0$ and $\infty$ (both of order $p$) and $1$ and $\alpha^{p}$ (both of order two). On the quotient orbifold $R/\langle A \rangle=S/H$ the involution $\iota$ induces a conformal involution $J$ permuting the branch values $\infty,0,1,\lambda_{1}$ in pairs. Up to the action of ${\mathbb G}_{3}$, we may assume $J(z)=\lambda_{1}/z$. A corresponding branched cover induced by $J$ is given by $Q_{3}(z)=(z^{2}-(\lambda_{1}+1)z+\lambda_{1})/(2\sqrt{\lambda_{1}}-\lambda_{1}-1)z)$. It sends the two fixed points of $J$ onto $1$ and $(2\sqrt{\lambda_{1}}+\lambda_{1}+1)/(\lambda_{1}+1-2\sqrt{\lambda_{1}})$ and the four cone points of order $p$ onto $\infty$ and $0$. This permits to obtain the desired relation between $\alpha$ and $\lambda_{1}$ as desired.

Working similarly as in the case $p=2$, if we write $G=\{1,u_{1}, u_{2}=u_{1}^{2}, \ldots, u_{p-1}=u_{1}^{p-1}\}$, then $K=K_{P}$ for some $P=(I_{1},\ldots,I_{p-1})$ such that
$\#I_{1}+\cdots+\#I_{p-1}=n+1$ and $\#I_{1}+2\#I_{2}+\cdots+(p-1)\#I_{p-1} \equiv 0 \mod (p)$. 

If $n=2$, we obtain $\#I_{1}=2$, $\#I_{p-2}=1$ and all others $I_{j}=\varnothing$. In this case,  $K=\langle a_{2}a_{1}^{-1}\rangle$. 

If $n=3$, we obtain $\#I_{1}=3$, $\#I_{p-3}=1$ and all others $I_{j}=\varnothing$. In this case, $K=\langle a_{2}a_{1}^{-1}, a_{3}a_{1}^{-1}\rangle$.

\section{Example: $(p,n)=(2,4)$ (classical Humbert curves)} \label{Ejemplo}
In this section, $S=C^{2}(\lambda_{1},\lambda_{2})$, and $H=H_{0}$, where $(\lambda_{1},\lambda_{2}) \in V_{4}$. Let 
$\{a_{1}, a_{2}, a_{3}, a_{4}, a_{5}\}$ be the set of standard generators of $H$ as in Section \ref{Sec:algebra}. In this case, $S$ has genus $g=5$ and (by Lemma \ref{cocientes})
the subgroups of $H$, acting freely and providing hyperelliptic quotients, are isomorphic to either ${\mathbb Z}_{2}$ or ${\mathbb Z}_{2}^{2}$.

\subsection{}
The $10$ subgroups of $H$, isomorphic to ${\mathbb Z}_{2}$ and acting freely on $S$, are given by
{\small
$$\begin{array}{lllll}
L_{1}=\langle a_{1}a_{2}\rangle,& L_{2}=\langle a_{1}a_{3}\rangle,& L_{3}=\langle a_{1}a_{4}\rangle,& L_{4}=\langle a_{1}a_{5}\rangle,& L_{5}=\langle a_{2}a_{3}\rangle,\\
L_{6}=\langle a_{2}a_{4}\rangle, & L_{7}=\langle a_{2}a_{5}\rangle, & L_{8}=\langle a_{3}a_{4}\rangle, & L_{9}=\langle a_{3}a_{5}\rangle, & L_{10}=\langle a_{4}a_{5}\rangle.
\end{array}
$$
}

The $10$ hyperelliptic curves of genus three, provided by these $10$ subgroups, are given by
{\small
$$y^{2}=(x^{4}+2(1-2a)x^{2}+1)(x^{4}+2(1-2b)x^{2}+1),$$
}
where $(a,b)$ runs over the following pairs
{\small
$$\begin{array}{l}
(\lambda_{1},\lambda_{2}),(1-\lambda_{1},\lambda_{2}(1-\lambda_{1})/(\lambda_{2}-\lambda_{1})),(\lambda_{1}/(\lambda_{1}-1),(\lambda_{2}-\lambda_{1})/(1-\lambda_{1})), \\
(1/\lambda_{1},\lambda_{2}/\lambda_{1}),  (1-\lambda_{2},\lambda_{1}(1-\lambda_{2})/(\lambda_{1}-\lambda_{2}),(\lambda_{2}/(\lambda_{2}-1),(\lambda_{1}-\lambda_{2})/(1-\lambda_{2})),\\
(1/\lambda_{2},\lambda_{1}/\lambda_{2}),
((1-\lambda_{1})/(1-\lambda_{2}),\lambda_{2}(1-\lambda_{1})/(\lambda_{1}(1-\lambda_{2}))),
(\lambda_{2}/\lambda_{1},(1-\lambda_{2})/(1-\lambda_{1})), \\ (\lambda_{1}/\lambda_{2},\lambda_{1}(1-\lambda_{2})/(\lambda_{2}(1-\lambda_{1}))).
\end{array}
$$
}
\subsection{}
The $10$ subgroups of $H$, isomorphic to ${\mathbb Z}_{2}^{2}$ and acting freely on $S$, are given by
$$\begin{array}{c}
K_{1}=\langle a_{1}a_{2},a_{1}a_{3}\rangle,\; K_{2}=\langle a_{1}a_{2},a_{1}a_{4}\rangle,\; K_{3}=\langle a_{1}a_{2},a_{1}a_{5}\rangle,\; K_{4}=\langle a_{1}a_{3},a_{1}a_{4}\rangle,\\
K_{5}=\langle a_{1}a_{3},a_{1}a_{5}\rangle,\;K_{6}=\langle a_{1}a_{4},a_{1}a_{5}\rangle, \; K_{7}=\langle a_{2}a_{3},a_{2}a_{4}\rangle,\\
K_{8}=\langle a_{2}a_{3},a_{2}a_{5}\rangle, \; K_{9}=\langle a_{2}a_{4},a_{2}a_{5}\rangle, \; K_{10}=\langle a_{3}a_{4},a_{3}a_{5}\rangle.
\end{array}
$$
In order to get algebraic curves descriptions, for the above corresponding $10$ Riemann surfaces of genus two, we proceed as follows. We consider the $10$ choices for $\{b_{1},b_{2}\}$: 
(i) $\{\infty,0\}$, (ii) $\{\infty,1\}$, (iii) $\{\infty,\lambda_{1}\}$, (iv) $\{\infty,\lambda_{2}\}$, (v) $\{0,1\}$, (vi) $\{0,\lambda_{1}\}$, (vii) $\{0,\lambda_{2}\}$, (viii) $\{1,\lambda_{1}\}$, (ix) $\{1,\lambda_{2}\}$, (x) $\{\lambda_{1},\lambda_{2}\}$. The choices for $Q(z)$ we may use in each case are: (i) $Q(z)=z^{2}$, (ii) $Q(z)=z^{2}+1$, (iii) $Q(z)=z^{2}+\lambda_{1}$,    (iv) $Q(z)=z^{2}+\lambda_{2}$,  (v) $Q(z)=1/(z^{2}+1)$, (vi) $Q(z)=\lambda_{1}/(z^{2}+1)$, (vii) $Q(z)=\lambda_{2}/(z^{2}+1)$, (viii) $Q(z)=(z^{2}+\lambda_{1})/(z^{2}+1)$,  (ix) $Q(z)=(z^{2}+\lambda_{2})/(z^{2}+1)$, (x) $Q(z)=(\lambda_{1}z^{2}+\lambda_{2})/(z^{2}+1)$. In this way, we obtain the $10$ desired hyperelliptic Riemann surfaces (in the first one, $C_{1}$, we have also changed $(x,y)$ by $(ix,iy)$):
{\fontsize{8}{8}\selectfont
$$
\begin{array}{l}
C_{1}: \; y^{2}=(x^{2}+1)(x^{2}+\lambda_{1})(x^{2}+\lambda_{2}),\;
C_{2}: \; y^{2}=(x^{2}+1)(x^{2}+1-\lambda_{1})(x^{2}+1-\lambda_{2}),\\
C_{3}: \; y^{2}=(x^{2}+\lambda_{1})\left(x^{2}-1+\lambda_{1}\right)\left(x^{2}-\lambda_{2}+\lambda_{1}\right),\;
C_{4}: \; y^{2}=(x^{2}+\lambda_{2})\left(x^{2}-1+\lambda_{2}\right)\left(x^{2}-\lambda_{1}+\lambda_{2}\right),\\
C_{5}: \; y^{2}=(x^{2}+1)\left(x^{2}+(\lambda_{1}-1)/\lambda_{1}\right)\left(x^{2}+(\lambda_{2}-1)/\lambda_{1}\right),\\
C_{6}: \; y^{2}=(x^{2}+1)\left(x^{2}+1-\lambda_{1}\right)\left(x^{2}+(\lambda_{2}-\lambda_{1})/\lambda_{2}\right),\;
C_{7}: \; y^{2}=(x^{2}+1)\left(x^{2}+1-\lambda_{2}\right)\left(x^{2}+(\lambda_{1}-\lambda_{2})/\lambda_{1}\right),\\
C_{8}: \; y^{2}=(x^{2}+1)(x^{2}+\lambda_{1})\left(x^{2}+(\lambda_{2}-\lambda_{1})/(1-\lambda_{2})\right),\\
C_{9}: \; y^{2}=(x^{2}+1)(x^{2}+\lambda_{2})\left(x^{2}+(\lambda_{1}-\lambda_{2})/(1-\lambda_{1})\right),\\
C_{10}: \; y^{2}=(x^{2}+1)\left(x^{2}+\lambda_{2}/\lambda_{1}\right)\left(x^{2}+(\lambda_{2}-1)/(\lambda_{1}-1)\right).
\end{array}
$$
}
Note that if we change $(x,y)$ by $\left(\sqrt{\lambda_{1}}x,\sqrt{\lambda_{1}^{3}}y\right)$, then $C_{3}$ is transformed into the curve
{\small
$$C'_{3}: \; y^{2}=(x^{2}+1)\left(x^{2}+(\lambda_{1}-1)/\lambda_{1}\right)\left(x^{2}+(\lambda_{1}-\lambda_{2})/\lambda_{1}\right),$$
}
and if we change $(x,y)$ by $\left(\sqrt{\lambda_{2}}x,\sqrt{\lambda_{2}^{3}}y\right)$, then $C_{4}$ is transformed into the curve
{\small
$$C'_{4}: \; y^{2}=(x^{2}+1)\left(x^{2}+(\lambda_{2}-1)/\lambda_{2}\right)\left(x^{2}+(\lambda_{2}-\lambda_{1})/\lambda_{2}\right).$$
}

\subsection{}
Each subgroup $K_{j}$ contains exactly $3$ of the subgroups $L_{k}$'s; for instance, $K_{1}$ contains $L_{1}$, $L_{2}$ and $L_{5}$.
As noted before, the genus two surface $S/K_{j}$ is obtained by considering two points $b_{1},b_{2} \in \{\infty,0,1,\lambda_{1},\lambda_{2}\}$. A Riemann surface $S/L_{k}$ over $S/H_{j}$ is obtained by considering a point $b_{3} \in \{\infty,0,1,\lambda_{1},\lambda_{2}\}-\{b_{1},b_{2}\}$. In this way, once we have chosen $b_{1}$ and $b_{2}$, there are exactly $3$ possible choices for $b_{3}$; these are the three subgroups $L_{k}$'s contained inside $K_{j}$.
For example, if we take $\{b_{1},b_{2}\}=\{\lambda_{1},\lambda_{2}\}$, 
then the genus two surface (uniformized by one of the $K_{j}$'s) is given by 
{\small
$$y^{2}=(x^{2}+1)\left(x^{2}+\lambda_{2}/\lambda_{1}\right)\left(x^{2}+(\lambda_{2}-1)/(\lambda_{1}-1)\right),$$
}
and the three genus three surfaces (uniformized by one of the $L_{k}$'s contained in the corresponding $K_{j}$) are 
{\small
$$
\begin{array}{ll}
y^{2}=(x^{4}+1)\left(x^{4}+\lambda_{2}/\lambda_{1}\right), &\mbox{if} \; b_{3}=1.\\
y^{2}=(x^{4}+\lambda_{2}/\lambda_{1})\left(x^{4}+(\lambda_{2}-1)/(\lambda_{1}-1)\right), & \mbox{if} \; b_{3}=\infty.\\
y^{2}=(x^{4}+1)\left(x^{4}+(\lambda_{2}-1)/(\lambda_{1}-1)\right), & \mbox{if} \; b_{3}=0.
\end{array}
$$
}

\section{A remark: connection to some parameter spaces}\label{prueba}
In order to state our next result, we need to recall some general facts on the complex analytic theory of Teichm\"uller and moduli spaces of Riemann orbifolds (good references are, for instance, \cite{Bers, Nag, Royden}).  

\subsection{Riemann orbifolds}
A Riemann orbifold of type $(g,r)$ is a pair ${\mathcal O}=(S,\{p_{1},\ldots,p_{r}\})$, where $S$ is a closed Riemann surface of genus $g$ and $p_{1},\ldots,p_{r} \subset S$ are pairwise distinct points (we allow to have $r=0$). In this case, $S$ is called the underlying Riemann surface of the orbifold and the points $p_{j}$ its cone points. The typical example of a Riemann orbifold is the one obtained as the quotient of a closed Riemann surface by a finite group of its conformal automorphisms (in this case, the cone points are the projection of the fixed points of the non-trivial elements of the group). The orbifold ${\mathcal O}$ is of hyperbolic type if $2g+r>2$ (in this case, it can be seen as a quotient ${\mathbb H}^{2}/\Gamma$, for a suitable Fuchsian group $\Gamma$ acting on the hyperbolic plane ${\mathbb H}^{2}$).

\subsection{Teichm\"uller and moduli spaces of Riemann orbifolds}
Let us fix a Riemann orbifold ${\mathcal O}=(S,\{p_{1},\ldots,p_{r}\})$.

A marking of ${\mathcal O}$ is a pair $(f,{\mathcal O}')$, where ${\mathcal O}'=(S',\{q_{1},\ldots,q_{r}\})$ is a Riemann orbifold of type $(g,r)$ and $f:S \to S'$ is an orientation-preserving homeomorphism which sends the cone points $\{p_{1},\ldots,p_{r}\}$ onto the set of cone points $\{q_{1},\ldots,q_{r}\}$. 

Two markings $(f_{1},{\mathcal O}_{1}=(S_{1},\{q_{1,1},\ldots,q_{1,r}\}))$ and $(f_{2},{\mathcal O}_{2}=(S_{2},\{q_{2,1},\ldots,q_{2,r}\}))$ are called equivalent if there is a biholomorphism $A:S_{1} \to S_{2}$ such that $A(f_{1}(p_{j}))=f_{2}(p_{j})$, for every $j=1,\ldots, r$, with $f_{2}^{-1} \circ A \circ f_{1}$ homotopic to the identity relative to the set $\{p_{1},\ldots,p_{r}\}$. The set of these equivalence classes is the Teichm\"uller space $T_{g,r}$. Results due to Fricke and Klein \cite{Fricke} assert that $T_{g,r}$ is a real topological manifold homeomorphic to ${\mathbb R}^{6g-6+2r}$. If $2g+r \geq 4$, then quasiconformal maps theory permits to provide to $T_{g,r}$ the structure of a complex manifold of dimension $3g-3+r$ \cite{Bers2}. If $(g,r) \in \{(1,0), (1,1)\}$, then this space has dimension $1$ and it can identified with the hyperbolic plane.

 The modular group of ${\mathcal O}$, denoted by ${\rm Mod}_{g,r}$, is the connected component of the identity map in the group of homotopy classes (relative to the set $\{p_{1},\ldots,p_{r}\}$) of orientation-preserving self-homeomorphisms of $S$ keeping invariant $\{p_{1},\ldots,p_{r}\}$. 
 It is known that ${\rm Mod}_{g,r}$ acts discontinuously as a group of holomorphic automorphisms of $T_{g,r}$. In fact, if $2g+r \geq 5$, then ${\rm Mod}_{g,r}$ is the full group of holomorphic (and orientation-preserving isometries) of $T_{g,r}$ (\cite{E-K,Royden}).
 
 The quotient ${\mathcal M}_{g,r}=T_{g,r}/{\rm Mod}_{g,r}$ is called the moduli space of orbifolds of type $(g,r)$ and it has the structure of a complex orbifold  of the same dimension as $T_{g,r}$. If $r=0$, then we set $T_{g}=T_{g,0}$ and ${\mathcal M}_{g}={\mathcal M}_{g,0}$.

\subsection{Certain subloci of moduli spaces}

\subsubsection{}
Let $S$ be a hyperelliptic Riemann surface of genus $g \geq 2$, with hyperelliptic  involution $\iota$ (which has exactly $2g+2$ fixed points and it is known to be unique). The quotient orbifold $S/\langle \iota \rangle$ can be identified with the Riemann orbifold $(\widehat{\mathbb C},\{p_{1},\ldots,p_{2g+2}\})$, where $\{p_{1},\ldots,p_{2g+2}\}$ is the projection set of the fixed points of $\iota$. The uniqueness of the hyperelliptic involution permits to observe that any two hyperelliptic Riemann surfaces are biholomorphically equivalent if and only if the quotient orbifolds (by their corresponding hyperelliptic involutions) are biholomorphic as orbifolds (i.e., there is a M\"obius transformation sending the cone points of the first onto the cone points of the second one). This permits to observe that the moduli space   ${\mathcal M}_{g}^{hyp}$, of hyperelliptic Riemann surfaces of genus $g \geq 2$, can be identified with the moduli space ${\mathcal M}_{0,2g+2}$ and, in particular, that it is a complex orbifold of dimension $2g-1$. 
The uniqueness of the hyperelliptic involution also asserts that there is a natural holomorphic embedding of ${\mathcal M}_{g}^{hyp}$ into ${\mathcal M}_{g}$. 

\subsubsection{}
If $g \geq 2$ is even, then we denote by ${\mathcal M}_{(g,2)}$ the sublocus of ${\mathcal M}_{g}$ consisting 
of those classes of Riemann surfaces admitting a conformal involution with exactly two fixed points. 

If  $S$ is a closed Riemann surface of even genus $g$, admitting a conformal involution $\tau$ with exactly two fixed points, then $S/\langle \tau \rangle$ is an orbifold of genus $g/2$ with exactly two cone points, each of order two. This provides a holomorphic embedding of the Teichm\"uller space $T_{g/2,2}$ into the Teichm\"uller space $T_{g}$. Such embedded space projects into ${\mathcal M}_{g}$ as a complex analytic space (a subset of ${\mathcal M}_{(g,2)}$) of dimension equal to the dimension of $T_{g/2,2}$, that is, $(3g-2)/2$. This analytic space has 
some analytic singularities; they correspond to Riemann surfaces of genus $g$ admitting two different involutions, each one with two fixed points, which are not conjugated in the group of holomorphic automorphisms of the surface. Its normalization is given by the quotient of the embedded $T_{g/2,2} \subset T_{g}$ by its stabilizer subgroup in $Mod_{g}$ (this normalization space is a complex manifold of dimension $(3g-2)/2$ which happens to be a finite branched cover of the moduli space ${\mathcal M}_{g/2,2}$).

The topological action of such type of involutions $\tau$ is unique, in the sense that if $R$ is another closed Riemann surface of genus $g$ admitting a conformal involution $\eta$ with exactly two fixed points, then there is a quasiconformal homeomorphism $f:S \to R$ conjugating $\tau$ into $\eta$. 
This uniqueness asserts that ${\mathcal M}_{(g,2)}$ coincides with the above analytic quotient space. 

Each connected component of ${\mathcal M}^{hyp}_{(g,2)}={\mathcal M}_{(g,2)} \cap {\mathcal M}_{g}^{hyp}$ (after normalization if necessary in the presence of singularities) can be identified with (a finite cover of) the moduli space ${\mathcal M}_{0,g+3}$, which has dimension $g$.

\subsubsection{}
Similarly as above, for $g \geq 1$ odd, we denote by ${\mathcal M}_{(g,4)}$ the sublocus of ${\mathcal M}_{g}$ consisting 
of those classes of Riemann surfaces admitting a conformal involution with exactly four fixed points. This is an analytic space (again with analytic singularities) of dimension $(3g-1)/2$. The normalization of this analytic space can be identified with (a finite cover of) the moduli space ${\mathcal M}_{(g-1)/2,4}$.
In this case, each connected component of ${\mathcal M}^{hyp}_{(g,4)}={\mathcal M}_{(g,4)} \cap {\mathcal M}_{g}^{hyp}$ (after normalization) can again be identified with (a finite cover of) the moduli space ${\mathcal M}_{0,g+3}$.

\subsubsection{}
If $g \geq 1$, then we denote by ${\mathcal M}_{(g;2,2)}$ the moduli space of Riemann orbifolds of genus $g$ with exactly two cone points of order two. This space is equivalent to the moduli space ${\mathcal M}_{g,2}$ and, in particular, it is a complex orbifold of dimension $3g-1$. For $g \geq 2$, we let ${\mathcal M}^{hyp}_{(g;2,2)}$ be its subloci consisting of the conformal classes of those Riemann orbifolds whose underlying Riemann surface is hyperelliptic and whose hyperelliptic involution permutes the two cone points (it does not fixes them). This analytic space has dimension $2g$.

\subsubsection{}
If $(p-1)(n-1)>2$, then the moduli space ${\mathcal H}_{p,n}$ of generalized Fermat curves of type $(p,n)$ is isomorphic to $V_{n}/{\mathbb G}_{n}$ (Proposition \ref{proposicion3}) and
it in turn is also isomorphic to ${\mathcal M}_{0,n+1}$. We set ${\mathcal H}_{n}={\mathcal H}_{2,n}$.

Let $(S,H)$ be a generalized Fermat pair of type $(p,n)$, where $n \geq 3$ and $(p-1)(n-1)>2$.
Let us consider the homotopy class of $H$ inside ${\rm Mod}_{g_{p,n}}$, which we still denote by $H$, and let ${\mathcal T}_{H}(S) \subset T_{g_{p,n}}$ be the locus of its fixed points. The uniqueness of the generalized Fermat group $H$ asserts that the projection of ${\mathcal T}_{H}(S)$ into ${\mathcal M}_{g_{p,n}}$ can be identified with the moduli space ${\mathcal H}_{p,n}$ (i.e., there is a natural holomorphic embedding of ${\mathcal H}_{p,n}$ into ${\mathcal M}_{g_{p,n}}$).

\subsection{Some connections between the above moduli spaces}

\begin{theo}\label{main}
\mbox{}
\begin{enumerate}[label=(\alph*),leftmargin=*,align=left]

\item[(1)] If $n \geq 4$ is an even integer, then 
\begin{itemize}
\item[(1.1)] there is a generically injective holomorphic map
${\mathcal M}_{0,n+1} \to \left({\mathcal M}^{hyp}_{(n-2,2)}\right)^{n(n+1)/2}$. 
\item[(1.2)] there is a degree $n(n+1)/2$ holomorphic surjective map
${\mathcal M}^{hyp}_{(n-2,2)} \to {\mathcal M}_{0,n+1}$.
\item[(1.3)] there is a generically injective holomorphic map
${\mathcal M}_{0,n+1} \to \left({\mathcal M}^{hyp}_{((n-2)/2;2,2)}\right)^{n+1}$.
\item[(1.4)] there is a degree $(n+1)$ holomorphic surjective map
${\mathcal M}^{hyp}_{((n-2)/2;2,2)} \to {\mathcal M}_{0,n+1}$.
\end{itemize}
\item[(2)] If $n \geq 5$ is an odd integer, then
\begin{itemize} 
\item[(2.1)] there is a generically injective holomorphic map
${\mathcal M}_{0,n+1} \to \left({\mathcal M}^{hyp}_{(n-2,4)}\right)^{n(n+1)/2}$.
\item[(2.2)] there is a degree $n(n+1)/2$ holomorphic surjective map
${\mathcal M}^{hyp}_{(n-2,4)}\to {\mathcal M}_{0,n+1}$.

\end{itemize}
\end{enumerate}
\end{theo}

\subsection{Proof of part (1) of Theorem \ref{main}}
We assume $(S,H)$ is a generalized Fermat pair of type $(2,n)$, where $n \geq 4$ is even and let $K_{1}$,\ldots, $K_{n(n+1)/2}$ be those subgroups of $H$ isomorphic to ${\mathbb Z}_{2}^{n-2}$ and acting freely on $S$. Denote, as before, by $a_{1},\ldots,a_{n+1}$ the standard generators of $H$. We already know  that $S/K_{i}$ is a hyperelliptic Riemann surface of genus $n-2$ and that $H/K_{i} <{\rm Aut}(S/K_{i})$ is generated by the hyperelliptic involution $j_{i}$ and a conformal involution $\tau_{i}$ with exactly two fixed points ($j_{i} \tau_{i}$ also has exactly two fixed points).  Part (1) of the following lemma asserts that, up to isomorphisms, in the above we obtain all possible pairs $(R,G)$, where $R$ runs over the hyperelliptic Riemann surfaces of genus $n-2$ and ${\mathbb Z}_{2}^{2} \cong G<{\rm Aut}(R)$ contains  the hyperelliptic involution of $R$.

\begin{lemm}\label{embed2}
Let $R$ be a hyperelliptic Riemann surface of genus $n-2$, where $n \geq 4$ is even, whose hyperelliptic involution is $j$.
\begin{enumerate}[label=(\alph*),leftmargin=*,align=left]

\item[(1)] If $G<{\rm Aut}(R)$ is so that $G \cong {\mathbb Z}_{2}^{2}$ contains $j$, then there is a generalized Humbert pair $(S,H)$ and a subgroup
${\mathbb Z}_{2}^{n-2} \cong K<H$ acting freely on $S$ so that 
$(R,G)$ is conformally equivalent to $(S/K,H/K)$.

\item[(2)]  If $u,v \in {\rm Aut}(R)$ are conformal involutions, both of them different from $j$, then $\langle u,j\rangle=\langle v,j\rangle$.
\end{enumerate}
\end{lemm}
\begin{proof}
(1) As $G$ contains the hyperelliptic involution, by the Riemann-Hurwitz formula, the quotient $R/G$ has genus zero and exactly $n+1$ cone points, each one of order two. Now, just take $(S,H)$ as a generalized Humbert pair such that $S/H=R/G$ and use the fact that $S$ is the highest abelian Galois branched cover of the orbifold $S/H$.

(2). Let us consider a $2$-fold branched cover $\pi:R \to \widehat{\mathbb C}$ (its deck group is generated by the hyperelliptic involution). Then, both $u$ and $v$ descends by $\pi$ to commuting conformal involutions, say $\widehat{u}$ and $\widehat{v}$, respectively. If $\widehat{u}=\widehat{v}$, then we are done. Let us assume we have $\widehat{u} \neq \widehat{v}$, that is, $\langle  \widehat{u}, \widehat{v} \rangle \cong {\mathbb Z}_{2}^{2}$. Up to a M\"obius transformation, we may assume
$\widehat{u}(z)=1/z$ and $\widehat{v}(z)=-z$. As we are assuming that $j \notin \{u,v,uv\}$, none of $u$, $v$ or $uv$ may have a common fixed point with $j$ (this because the stabilizer of any point in ${\rm Aut}(R)$ is cyclic). It follows that none of $\widehat{u}$, $\widehat{v}$ or $\widehat{u}\widehat{v}$ fixes a branch value of $\pi$ and, in particular, that $R$ must have a curve representation as follows
{\small
$$y^{2}=\prod_{j=1}^{(n-1)/2} \left( x^{2}-a_{j}^{2} \right) \left( x^{2}-a_{j}^{-2}\right)$$
}
and $n$ is odd, a contradiction to the fact that $n$ was assumed to be even.
\end{proof}

\subsubsection{Proof of Parts (1.1) and (1.2)}
As the generic orbifold $S/H$ has trivial group of orbifold automorphisms, Lemma \ref{teo5} asserts that the $n(n+1)/2$ pairs
$$(S/K_{1},H/K_{1}),..., (S/K_{n(n+1)/2},H/K_{n(n+1)/2})$$
are generically pairwise conformally non-equivalent. Now, part (2) of Lemma \ref{embed2} asserts that the hyperelliptic Riemann surfaces $S/K_{1},\ldots, S/K_{n(n+1)/2}$ are generically pairwise conformally non-equivalent, in particular, 
$${\mathcal M}_{0,n+1} \to \left({\mathcal M}^{hyp}_{(n-2,2)}\right)^{n(n+1)/2}:
[(S,H)] \mapsto ([S/K_{1},H/K_{1})],\ldots, [(S/K_{n(n+1)/2},H/K_{n(n+1)/2})])$$
is a generically injective holomorphic map. This provides Part (1.1) of Theorem \ref{main}. 

Part (1.2) of Theorem \ref{main} is just a consequence of Part (1.1) and Part (1) of Lemma \ref{embed2}.
We proceed to describe the desired surjective holomorphic map in terms of the domain $V_{n}$, where
$$V_{n}=\{(\lambda_{1},\ldots,\lambda_{n-2}) \in ({\mathbb C} \setminus \{0,1\})^{n-2}: \lambda_{i} \neq \lambda_{j}, \; i \neq j\} \subset {\mathbb C}^{n-2}.$$

Assume we are given a hyperelliptic Riemann surface $R$ of genus $(n-2)$, whose hyperelliptic involution is $j$, and $G=\langle j,\tau\rangle \cong {\mathbb Z}_{2}^{2}$, a group of conformal automorphism of $R$, so that $\tau$ has exactly two fixed points ($j \tau$ also has exactly two fixed points) and $R/G$ is an orbifold of genus zero and exactly $n+1$ cone points, each one of order two. We may assume $R/G$ is the Riemann sphere and the conical points are $\infty, 0, 1, \lambda_{1},\ldots, \lambda_{n-2}$, so that (i) $\lambda_{n-3}$ is the projection of both fixed points of $\tau$ and (ii) $\lambda_{n-2}$ is the projection of both fixed points of $j \tau$. This choice is not unique as we may compose at the left by a M\"obius transformation that sends any of three points in $\{\infty,0,1,\lambda_{1},\ldots,\lambda_{n-4}\}$ to $\infty$, $0$ and $1$. This  corresponds to the action on $V_{n}$ by the group ${\mathfrak S}_{n-1}=\langle s,b\rangle,$
where
{\small
$$s(\lambda_{1},\ldots,\lambda_{n-2})= \left( \frac{\lambda_{n-4}}{\lambda_{n-4}-1},\frac{\lambda_{n-4}}{\lambda_{n-4}-\lambda_{1}},\ldots, \frac{\lambda_{n-4}}{\lambda_{n-4}-\lambda_{n-5}},
\frac{\lambda_{n-4}}{\lambda_{n-4}-\lambda_{n-3}},\frac{\lambda_{n-4}}{\lambda_{n-4}-\lambda_{n-2}} \right),$$
$$b(\lambda_{1},\ldots,\lambda_{n-2})=\left( \frac{1}{\lambda_{1}},\ldots,\frac{1}{\lambda_{n-2}}\right).$$

}
Next, as we may permute the involutions $\tau$ and $j\tau$, we also need to consider the action of the involution
$c(\lambda_{1},\ldots,\lambda_{n-2})=(\lambda_{1},\ldots,\lambda_{n-4},\lambda_{n-2},\lambda_{n-3}).$
Note that $cs=sc$ and $cb=bc$, so $\langle {\mathfrak S}_{n-1},c\rangle= {\mathfrak S}_{n-1} \oplus {\mathbb Z}_{2}$. As a consequence of the above (together  with Lemma \ref{embed2}), a model of the space ${\mathcal M}^{hyp}_{(n-2,2)}$ is given by 
$V_{n}/({\mathfrak S}_{n-1} \oplus {\mathbb Z}_{2}).$
Also, a model of the moduli space of pairs $(R,\tau)$, where $R$ is a hyperelliptic  Riemann surface of genus $n-2$ and $\tau:R \to R$ is a conformal involution with exactly two fixed points, is given by $V_{n}/{\mathfrak S}_{n-1}$. In these models, the surjective holomorphic map in Part (1.2) of Theorem \ref{main} corresponds to the canonical projection map 
$$V_{n}/({\mathfrak S}_{n-1} \oplus {\mathbb Z}_{2}) \to V_{n}/{\mathfrak S}_{n+1}$$
in the following diagram
$$
\xymatrixcolsep{4pc}
\xymatrix{ 
V_{n} \ar[r]^{{\mathfrak S}_{n-1}} & V_{n}/{\mathfrak S}_{n-1} \ar[r]^{{\mathfrak S}_{n-1} \oplus {\mathbb Z}_{2}} \ar[rd]^{n(n+1)} &V_{n}/({\mathfrak S}_{n-1} \oplus {\mathbb Z}_{2}) \ar[d]^{\frac{n(n+1)}{2}} \\ 
& &{\mathcal M}_{0,n+1}}
$$

\begin{exem}[$n=4$]
Let $(\lambda_{1},\lambda_{2}) \in V_{4}$ be so that $S/H$ is conformally equivalent to the orbifold provided by $\widehat{\mathbb C}$ with conical points $\infty$, $0$, $1$, $\lambda_{1}$ and $\lambda_{2}$. Choose the conical points $\lambda_{1}$  and $\lambda_{2}$ and set $P(z)=(\lambda_{1}z^{2}+\lambda_{2})/(z^{2}+1)$. Then $P:\widehat{\mathbb C} \to \widehat{\mathbb C}$ is the branched covering of degree two with cover group generated by $\eta(z)=-z$ and branch values at $\lambda_{1}$ and $\lambda_{2}$. In this case $P^{-1}(\infty)=\pm i$, $P^{-1}(0)=\pm i \sqrt{\lambda_{2}/\lambda_{1}}$ and $P^{-1}(1)=\pm i \sqrt{(\lambda_{2}-1)/(\lambda_{1}-1)}$. These $6$ points define the hyperelliptic curve
{\small
$$C_{\lambda_{1},\lambda_{2}}: \; y^{2}=(x^{2}+1)\left(x^{2}+\lambda_{2}/\lambda_{1}\right)\left(x^{2}+(\lambda_{2}-1)/(\lambda_{1}-1)\right).$$
}

The curve $C_{\lambda_{1},\lambda_{2}}$ is one of the $10$ genus two Riemann surfaces uniformized by one of the acting freely subgroups $K_{j}$. The action of ${\mathfrak S}_{3} \oplus {\mathbb Z}_{2}$ at this level is given by:
{\small
$$
s: \; C_{\lambda_{1},\lambda_{2}} \mapsto C_{\frac{1}{1-\lambda_{1}},\frac{1}{1-\lambda_{2}}}:\; y^{2}=(x^{2}+1)\left(x^{2}+\frac{\lambda_{1}-1}{\lambda_{2}-1}\right)\left(x^{2}+\frac{\lambda_{2}(\lambda_{1}-1)}{\lambda_{1}(\lambda_{2}-1)}\right)
$$
$$
b: \; C_{\lambda_{1},\lambda_{2}} \mapsto C_{\frac{1}{\lambda_{1}},\frac{1}{\lambda_{2}}}:\; y^{2}=(x^{2}+1)\left(x^{2}+\frac{\lambda_{1}}{\lambda_{2}}\right)\left(x^{2}+\frac{\lambda_{1}(\lambda_{2}-1)}{\lambda_{2}(\lambda_{1}-1)}\right)
$$
$$
c: \; C_{\lambda_{1},\lambda_{2}} \mapsto C_{\lambda_{2},\lambda_{1}}: \; y^{2}=(x^{2}+1)\left(x^{2}+\frac{\lambda_{1}}{\lambda_{2}}\right)\left(x^{2}+\frac{\lambda_{1}-1}{\lambda_{2}-1}\right)
$$
}
\end{exem}

\subsubsection{Proof of Parts (1.3) and (1.4)}
 Any subgroup $L<H$ isomorphic to ${\mathbb Z}_{2}^{n-1}$ that contains some $K_{k}$ is of the form 
$L=\langle K_{k},a_{j}\rangle$,
for some standard generator $a_{j}$ of $H$. Up to permutation of indices, we may assume
$K_{k}=\langle a_{1}a_{2}, a_{1}a_{3},\ldots,a_{1}a_{n-1}\rangle.$ 
If $j \in \{1,2,\ldots,n-1\}$, then
$L=\langle K_{k},a_{j} \rangle=\langle a_{1},a_{2},\ldots,a_{n-1}\rangle$ and $H/L$ is the cyclic group generated by the hyperelliptic involution of $S/K_{k}$. We call any of this kind of subgroup $L$ a {\it hyperelliptic-${\mathbb Z}_{2}^{n-1}$-subgroup of $H$}. The following is now clear.

\begin{theo}\label{teo8}
If $(S,H)$ is a generalized Fermat pair of type $(2,n)$, where $n \geq 4$ is even, then the number of different hyperelliptic-${\mathbb Z}_{2}^{n-1}$-subgroups of $H$ is $n(n+1)/2$.
\end{theo}

Let us now consider the case $j \in \{n,n+1\}$.
The two different groups 
$L_{1}=\langle K_{k},a_{n} \rangle$ and $L_{2}=\langle K_{k},a_{n+1} \rangle$
have the property that $H/L_{j}$ is generated by a conformal involution (different from the hyperelliptic one) of $S/K_{k}$ having exactly $2$ fixed points. In this way, $S/L_{j}$ is an orbifold of signature $((n-2)/2;2,2)$. We call this kind of groups $L_{j}$ a {\it non-hyperelliptic-${\mathbb Z}_{2}^{n-1}$-subgroup of $H$}. At this point, we note that, as there are exactly $n(n+1)/2$ different possibilities for $K_{k}$, there are at most $n(n+1)$ different non-hyperelliptic-${\mathbb Z}_{2}^{n-1}$-subgroups of $H$.

\begin{lemm}\label{lemma2}
Let $(S,H)$ be a generalized Fermat pair of type $(2,n)$, where $n \geq 4$ is even, and let $a_{1},\ldots, a_{n+1}$ be the standard generators of $H$.
Let $j_{1},\ldots,j_{n-1},k_{1},\ldots,k_{n-1} \in \{1,2,\ldots,n+1\}$ be so that
$j_{1},\ldots,j_{n-1}$ (respectively, $k_{1},\ldots,k_{n-1}$) are pairwise different. 

If $U_{1}=\langle a_{j_1}a_{j_2}, a_{j_1}a_{j_3},\ldots,a_{j_1}a_{j_{n-1}}\rangle$,  
$U_{2}=\langle a_{k_1}a_{k_2}, a_{k_1}a_{k_3},\ldots,a_{k_1}a_{k_{n-1}}\rangle$ and 
$a_{r} \in \{1,\ldots,n+1\} \setminus \{j_{1},\ldots,j_{n-1},k_{1},\ldots,k_{n-1}\}$, then
$\langle U_{1},a_{r}\rangle=\langle U_{2},a_{r}\rangle$.
\end{lemm}
\begin{proof}
We may assume, up to permutation of indices, that
$U_{1}=\langle a_{1}a_{2}, a_{1}a_{3},\ldots,a_{1}a_{n-1}\rangle$ and $r=n+1$. As 
$a_{1}a_{n}a_{n+1}=(a_{1}a_{2})(a_{1}a_{3})\cdots(a_{1}a_{n-1}) \in U_{1}$, it follows that 
$a_{1}a_{n} \in \langle U_{1},a_{n+1}\rangle$, in particular, that $a_{i}a_{j} \in \langle U_{1},a_{n+1}\rangle$, for all $i,j \in \{1,2,\ldots,n\}$. This ensures $\langle U_{2}, a_{n+1}\rangle < \langle U_{1},a_{n+1}\rangle$ and, in particular, that they are equal.
\end{proof}

As a consequence of the previous Lemma, we obtain.

\begin{theo}\label{teo9}
Let $(S,H)$ be a generalized Fermat pair of type $(2,n)$, where $n \geq 4$ is even. Then,
there are exactly $n+1$ different non-hyperelliptic-${\mathbb Z}_{2}^{n-1}$-subgroups of $H$.
\end{theo}

Now, let $L_{1},\ldots, L_{n+1}<H$ be the $n+1$ different non-hyperelliptic-${\mathbb Z}_{2}^{n-1}$-subgroups of $H$. Again, as for generic pair $(S,H)$ we have that $S/H$ has trivial orbifold automorphism group, generically the $n+1$ orbifolds $S/L_{1},\ldots, S/L_{n+1}$ (each one of signature $((n-2)/2;2,2)$) are pairwise conformally non-equivalent. In particular, it follows that 
$${\mathcal M}_{0,n+1} \to \left({\mathcal M}^{hyp}_{((n-2)/2;2,2)}\right)^{n+1}:
[(S,H)] \to ([S/L_{1},H/L_{1}],\ldots,[S/L_{n+1},H/L_{n+1}])$$
is a generically injective holomorphic map, obtaining Part (1.3) of Theorem \ref{main}.
As a generalized Fermat curve of type $(2,n)$ is the homology covering of an orbifold of genus zero with all of its cone points of order $2$, it follows Part (1.4) of Theorem \ref{main}.

\begin{rema}
In order to get equations for the underlying hyperelliptic Riemann surfaces $S/L_{j}$, we only need to choose one of the conical points of $S/H$ and consider the hyperelliptic Riemann surface determined by the other $n$ conical points. For example, if 
$n=4$ and $(\lambda_{1},\lambda_{2}) \in V_{4}$, then, up to equivalence, the $n+1=5$ curves of genus one are given by
{\small
$$\begin{array}{c}
y^{2}=x(x-1)(x-\lambda_{1}(\lambda_{2}-1)/\lambda_{2}(\lambda_{1}-1)), \quad
y^{2}=x(x-1)(x-(\lambda_{2}-1(/(\lambda_{1}-1)), \\
y^{2}=x(x-1)(x-\lambda_{1}/\lambda_{2}), \quad
y^{2}=x(x-1)(x-\lambda_{1}),\quad
y^{2}=x(x-1)(x-\lambda_{2}).
\end{array}
$$
}
\end{rema}

\subsection{Proof of part (2) of Theorem \ref{main}}
Let us now assume $(S,H)$ is a generalized Fermat pair of type $(2,n)$, where $n \geq 5$ is odd, and that $K_{1},\ldots, K_{n(n+1)/2}$ are those subgroups isomorphic to ${\mathbb Z}_{2}^{n-2}$ acting freely on $S$. We may proceed as in the even case and to obtain the commutative diagram
$$
\xymatrixcolsep{4pc}
\xymatrix{ 
 V_{n} \ar[r]^{{\mathfrak S}_{n-1}} \ar[rd]^{{\mathfrak S}_{n+1}} &V_{n}/{\mathfrak S}_{n-1} \ar[d]^{n(n+1)} \\ 
 &{\mathcal M}_{0,n+1}}
$$
\noindent
where $V_{n}/{\mathfrak S}_{n-1}$ is a model for the moduli space of hyperelliptic Riemann surfaces admitting a conformal involution with exactly $4$ fixed points. The proofs of Parts (2.1) and (2.2) follows the same lines as the previous cases.

\subsection*{Acknowledgements}
The author would like to thank both referees for their valuable comments, suggestions and corrections which helped to improve the paper.


\end{document}